\documentclass[11pt]{article}

\usepackage{epsf,latexsym,amsfonts,amsbsy}
\usepackage{abstract,amsmath, amssymb, amsthm,amscd,amsxtra}
\usepackage{graphicx}
\usepackage{xcolor}
\usepackage{yhmath}
\usepackage{algorithm}
\usepackage{algorithmicx}
\usepackage{algpseudocode}
\usepackage{tikz}
\usetikzlibrary{calc}
\usetikzlibrary{intersections,decorations.markings}
\usepackage{booktabs}
\usepackage{threeparttable}
\usepackage{pgfplots}
\pgfplotsset{compat=1.17}
\usepackage{caption}
\usepackage{hyperref} 
\hypersetup{
	colorlinks=true,
	linkcolor=blue,
	anchorcolor=blue,
	citecolor=blue}
\makeatletter % '@' is now a normal "letter" for TeX

\@addtoreset{equation}{section}

\@addtoreset{figure}{section}

\@addtoreset{table}{section}
\makeatother % '@' is restored as a "non-letter" character for TeX
\newtheorem{theorem}{Theorem}[section]

\newtheorem{definition}[theorem]{Definition}

\newtheorem{corollary}[theorem]{Corollary}

\newtheorem{lemma}[theorem]{Lemma}

\newtheorem{remark}[theorem]{Remark}
\newtheorem{example}[theorem]{Example}

\newcommand{\rank}{\mbox{\rm rank}\, }

\newcommand{\diag}{\mbox{\rm diag}\, }

\newenvironment{breakablealgorithm}
{% \begin{breakablealgorithm}
		\begin{center}
			\refstepcounter{algorithm}% New algorithm
			\hrule height.8pt depth0pt \kern2pt% \@fs@pre for \@fs@ruled
			\renewcommand{\caption}[2][\relax]{% Make a new \caption
				{\raggedright\textbf{\textbf{ Algorithm}~\thealgorithm} ##2\par}%
				\ifx\relax##1\relax % #1 is \relax
				\addcontentsline{loa}{algorithm}{\protect\numberline{\thealgorithm}##2}%
				\else % #1 is not \relax
				\addcontentsline{loa}{algorithm}{\protect\numberline{\thealgorithm}##1}%
				\fi
				\kern2pt\hrule\kern2pt
			}
		}{% \end{breakablealgorithm}
		\kern2pt\hrule\relax %\@fs@post% for \@fs@ruled
	\end{center}
}

\newenvironment{breakablealgorithm1}
{% \begin{breakablealgorithm1}
		\begin{center}
			\refstepcounter{algorithm}% New algorithm
			\hrule height.8pt depth0pt \kern2pt% \@fs@pre for \@fs@ruled
			\renewcommand{\caption}[2][\relax]{% Make a new \caption
				{\raggedright\textbf{\textbf{ Subalgorithm}~\thealgorithm} ##2\par}%
				\ifx\relax##1\relax % #1 is \relax
				\addcontentsline{loa}{algorithm}{\protect\numberline{\thealgorithm}##2}%
				\else % #1 is not \relax
				\addcontentsline{loa}{algorithm}{\protect\numberline{\thealgorithm}##1}%
				\fi
				\kern2pt\hrule\kern2pt
			}
		}{% \end{breakablealgorithm1}
		\kern2pt\hrule\relax %\@fs@post% for \@fs@ruled
	\end{center}
}

\newlength{\bibitemsep}
\newlength{\bibparskip}
\let\oldthebibliography\thebibliography
\renewcommand\thebibliography[1]{%
	\oldthebibliography{#1}%
	\setlength{\parskip}{\bibitemsep}%
	\setlength{\itemsep}{\bibparskip}%
}

\title{New results on the local-nonglobal minimizers of the generalized trust-region subproblem \thanks{This work was supported by the National Natural Science Foundation of China (Grants Nos. 12171052, 12171051). This work was also supported by Beijing Natural Science Foundation (Grants No. Z220004).}}

\author{Wenbao Ai\footnotemark[2] , Mengxiao Zhang\footnotemark[2] , Jianhua Yuan\thanks{with School of Science, Beijing University of Posts and Telecommunications, and also with Key Laboratory of Mathematics and Information Networks, Ministry of Education, Beijing 100876, China 
({aiwb@bupt.edu.cn, zhangmengxiao@bupt.edu.cn, jianhuayuan@bupt.edu.cn}). Corresponding author's Email: aiwb@bupt.edu.cn}}
\date{}

\begin{document}
\maketitle

\begin{abstract}
	In this paper, we study the local-nonglobal minimizers of the Generalized Trust-Region subproblem $(GTR)$ and its Equality-constrained version $(GTRE)$. Firstly, the equivalence is established between the local-nonglobal minimizers of both $(GTR)$ and $(GTRE)$ under assumption of the joint definiteness. By the way, a counterexample is presented to disprove a conjecture of Song-Wang-Liu-Xia [SIAM J. Optim., 33(2023), pp.267-293]. Secondly, if the Hessian matrix pair is jointly positive definite, it is proved that each of $(GTR)$ and $(GTRE)$ has at most two local-nonglobal minimizers. This result first confirms the correctness of another conjecture of Song-Wang-Liu-Xia [SIAM J. Optim., 33(2023), pp.267-293]. Thirdly, if the Hessian matrix pair is jointly negative definite, it is verified that each of $(GTR)$ and $(GTRE)$ has at most one local-nonglobal minimizer. In special, if the constraint is reduced to be  a ball or a sphere, the above result is just the classical Mart\'{i}nez's. Finally, an algorithm is proposed to find all the local-nonglobal minimizers of $(GTR)$ and $(GTRE)$ or confirm their nonexistence in a tolerance. Preliminary numerical results demonstrate the effectiveness of the algorithm.
\end{abstract}


{\bf Mathematics Subject Classiﬁcation }
90C20, 90C26, 90C32, 90C46

\section{Introduction}
Let us consider the following Generalized Trust-Region (GTR) subproblem
\begin{equation} \label{eq:inequalityproblem}
	\begin{array}{lcl}
		(GTR)&\underset{x\in\mathcal{R}^n}{\min} &f_0(x)=x^TA_0x+2b_0^Tx\\
		&\mbox{s.t.} &f_1(x)=x^TA_1x+2b_1^Tx+c_1\leq0,
	\end{array}
\end{equation}
and its Equality-constrained (GTRE) version:
\begin{equation} \label{eq:equalityproblem}
	\begin{array}{lcl}
		(GTRE)&\underset{x\in\mathcal{R}^n}{\min}   &f_0(x)=x^TA_0x+2b_0^Tx\\
		&\mbox{s.t.} &f_1(x)=x^TA_1x+2b_1^Tx+c_1=0,
	\end{array}
\end{equation}
where $A_0$, $A_1 \in \mathcal{S}^{n}$, $b_{0}$, $b_{1}\in\mathcal{R}^n$ and $c_{1} \in \mathcal{R}$. According to S-lemmas about inequality \cite{polik2007, yakubovich1971} and equality \cite{xia2016}, the problems $(GTR)$ and $(GTRE)$ have strong duality under mild conditions. So one can easily find their global minimizers by solving their Lagrangian dual problems.  In this paper we shall deeply study their local-nonglobal minimizers, because they play an important role in globally solving  nonconvex quadratically multi-constrained quadratic programming as follows:
\begin{equation}\label{eq:GTRadd1}
	\underset{x\in\mathcal{R}^n}{\min}\{f_0(x)\,\,|\,\,f_1(x)\leq0,\,f_2(x)\leq0,\cdots,f_m(x)\leq0\},
\end{equation}
where $m\geq2$ and  each $f_i$ is a quadratic (or linear) function of $x$ for $i=2,3,\cdots,m$. Specifically, all the local-nonglobal minimizers of \eqref{eq:inequalityproblem} that satisfy the constraints $f_i(x)\leq0$ for all $i\in\{2,3,\cdots,m\}$ are also local minimizers of \eqref{eq:GTRadd1}, and one of them may possibly be a global minimizer of \eqref{eq:GTRadd1}  when the hard case occurs that the Hessian of the Lagrangian of \eqref{eq:GTRadd1} has at least one negative eigenvalue at all Karush–Kuhn–Tucker points.
It should be indicated that,  when $m=2$ and $f_2(x)$ is a strictly convex quadratic function, the model \eqref{eq:GTRadd1} is just the famous generalized Celis-Dennis-Tapia subproblem (some newest advances involving the CDT subproblem may be found in \cite{bienstock2016, consolini2023, yuan-ai2017, yuan2015}).  
And when each $f_i$ is a linear function of $x$ ($i=2,3,\cdots,m$), the model \eqref{eq:GTRadd1} is referred to as generalization of the extended trust-region subproblem \cite{jiang2022}, which has applications in financial engineering \cite{luo2024}.
	
%which was extensively analyzed by researchers \cite{ai2009, beck2006, bienstock2016, bomze2015, burer2013, chen1999, chen2009,  grigoriev2005, jeyakumar2009, jeyakumar2014, peng1997, yang2016, ye2003, yuan1990, ai2017}.

%%%%In comparison with the global minimizers of %$(GTR)$ and $(GTRE)$, research on their local-nonglobal minimizers seems to encounter added difficulties. And
	
If $A_1\succ0$, $(GTR)$ and $(GTRE)$ are substantially the trust-region subproblem and its equality-constrained version, respectively. Characteristics and computational methods of their local-nonglobal minimizers have been deeply studied \cite{fortin2005, lucidi1998, martinez1994, xia2022, xia2020}. The first significant result was contributed by Mart\'{i}nez \cite{martinez1994} in 1994, in which it was proved that the trust-region subproblem and its equality-constrained version have at most one local-nonglobal minimizer. In 1998,  Lucidi, Palagi and Roma \cite{lucidi1998} indicated further that, for the trust-region subproblem, the corresponding multiplier to the local-nonglobal minimizer must be positive. In 2005, two new algorithms computing the local-nonglobal minimizer of the trust-region subproblem were proposed by Fortin \cite{fortin2005}. In 2020, a necessary and sufficient optimality condition was presented by Wang and Xia \cite{xia2020}, which eliminates the gap between the necessary and sufficient conditions for the local-nonglobal minimizer of the trust-region subproblem. Recently (in 2022), Wang, Song and Xia \cite{xia2022} showed that the local-nonglobal minimizer of the trust-region subproblem has the second smallest objective function value among all KKT points. 

The results on the local-nonglobal minimizer of the trust-region subproblem and its equality-constrained version have generated wide applications in theoretical analysis and algorithm design for some interesting special cases of \eqref{eq:GTRadd1} (see \cite{beck2017, bienstock2014, bomze2015, deng2020, xia2013, rontsis2022, salahi2017}). For example, Beck and Bienstock et al. \cite{beck2017, bienstock2014} considered the following problem
\begin{equation} \label{eq:s1-robust-model}
\underset{x\in\mathcal{R}^n}{\min}\{f_0(x)\,\,|\,\,\|x-a_i\|^2\leq d_i^2, i\in{I_1};\, 
\|x-a_i\|^2\geq d_i^2, i\in{I_2};\,
a_i^Tx\leq d_i, i\in{I_3}\},
\end{equation}
where $a_i\in\mathcal{R}^n$ and $d_i\in\mathcal{R}$ for each index $i$. Based on the result of Mart\'{i}nez \cite{martinez1994}, Bienstock and Michlka \cite{bienstock2014} proved that,  for each fixed pair $|I_1|$ and $|I_2|$, \eqref{eq:s1-robust-model} can be solvable in polynomial time by enumerating all the candidates for its global minimizers under certain mild conditions. Moreover, Beck and Pan \cite{beck2017} proposed a branch and bound algorithm for \eqref{eq:s1-robust-model}, in which one needs to find all the global and nonglobal minimizers of $f_0(x)$ on a sphere or a low-dimensional sphere. Recently, Rontsis et al. \cite{rontsis2022} presented an active-set algorithm for a special case of \eqref{eq:s1-robust-model}, which needs also to compute all the global and nonglobal minimizers of $f_0(x)$ on a sphere or a low-dimensional sphere.

%The problems $(GTR)$ and $(GTRE)$ have wide applications. When $A_{1}$ is positive definite, $(GTR)$ is essentially a trust-region subproblem, which plays a crucial role in the trust-region method and so has been widely researched. Some recent results about globally solving the trust-region subproblem may be found in \cite{adachi2017, beck2018, hazan2016, yuan2015}. When $A_{1}$ is a general symmetric matrix, back in 1971 the strong duality of $(GTR)$ was actually confirmed by the well-known S-lemma of Yakubovich \cite{polik2007, yakubovich1971} under the Slater's condition. Moreover, in 1993 Mor\'{e} \cite{more1993} proved that $(GTRE)$ also has the strong duality under two assumptions that $f_1$ is indefinite (i.e. $\inf\{f_1(x)\,|\,x\in\mathcal{R}^n\}<0<\sup\{f_1(x)\,|\,x\in\mathcal{R}^n\}$) and $A_1\ne0$. In 2016, the second assumption ``$A_1\ne0$" was improved by Xia, Wang and Sheu \cite{xia2016} into an exact condition as follows: the two statements that ``$A_1=0$" and ``$A_0$ has exactly one negative eigenvalue" do not hold simultaneously. Some approaches about solving globally $(GTR)$ and $(GTRE)$ were developed under the assumption that $A_{0}$ and $A_{1}$ can be simultaneously diagonalizable  or even have a joint positive definiteness \cite{adachi2019, ben-tal2014, jiang2019, jiang2020, jiang2018, wang2022, wang2023}.
%%\begin{redtext}

If $A_1\not\succ0$,  not much progress has only been made recently in the local-nonglobal minimizers of $(GTR)$ and  $(GTRE)$.  In 2020, Taati and Salahi \cite{taati2020} tried to characterize the local-nonglobal minimizers of $(GTR)$ under assumption of the joint positive definiteness, and proved that the number of the local-nonglobal minimizers of $(GTR)$ is no more than $2n+1$. In 2021 and 2023, Xia's team \cite{xia2023, xia2021} further proved that, if $(GTR)$ (or $(GTRE)$) satisfies three assumptions that ``$A_1\ne0$", ``the Slater's condition (or the two-side Slater's condition) holds" and ``the joint definiteness holds", then the local-nonglobal minimizers must also satisfy the sufficient optimality conditions and their number is no more than $\min\{\rank(A_1)+1, n\}$.  Particularly, the homogeneous $(GTR)$ and $(GTRE)$ have no local-nonglobal minimizers if only ``$c_1\ne0$".  Moreover, Xia's team conjectured in \cite[Conjecture 5.6]{xia2023} that $(GTR)$ has at most two local-nonglobal minimizers under two assumptions that ``the Slater's condition holds" and `` $A_0+\mu_1 A_1\succ0$ for some $\mu_1\geq0$". 

In this paper we focus on the local-nonglobal minimizers of $(GTR)$ and $(GTRE)$, and try to address some challenging issues, such as their number and how to efficiently compute them. The main contributions of this paper can be summarized as follows:
\begin{itemize}
	\item   The equivalence  is established between the local-nonglobal minimizers of $(GTR)$ and $(GTRE)$ if only $n\geq2$ and $\mu_0A_0+\mu_1 A_1\succ0$ for some $\mu_0,\mu_1\in\mathcal{R}$ (which we call the joint definiteness). 
	
	\item  It is proved that, if $A_0+\mu_1 A_1\succ0$ for some $\mu_1\in\mathcal{R}$ (which we call the joint positive definiteness), $(GTRE)$ (or $(GTR)$) has at most two local-nonglobal minimizers. So Conjecture 5.6 of \cite{xia2023} is confirmed to be true. 
	
	\item It is proved that, if $A_0+\mu_1 A_1\prec0$ for some $\mu_1\in\mathcal{R}$ (which we call the joint negative definiteness),  $(GTRE)$ (or $(GTR)$) has at most one local-nonglobal minimizer. In special, if $A_1$ is an identity matrix, the above result is just the classical Mart\'{i}nez's \cite{martinez1994}.
	
	\item An efficient algorithm is proposed to find all local-nonglobal minimizers of $(GTRE)$ (or $(GTR)$) or confirm their nonexistence in a tolerance.
	
	\item It is demonstrated by a counterexample that Conjecture 5.10 of \cite{xia2023} is not true (see Remark \ref{rmk:s2-cexmpl-2}).
\end{itemize}
We believe that our new results will help researchers to deal with the hard case of the model \eqref{eq:s1-robust-model}.

This paper is organized as follows. In Section \ref{sec-properties}, we study basic properties about the local-nonglobal minimizers of $(GTR)$ and $(GTRE)$.  In special, a counterexample is presented to show that Conjecture 5.10 of \cite{xia2023} is not true. In Section \ref{sec-number}, the theory involving the number of local-nonglobal minimizers is perfectly established. In Section \ref{sec-computation}, a bisection-based algorithm is proposed to find out all local-nonglobal minimizers, and preliminary numerical results are showed. 

\textbf{Notation:} Let $S^{n}$ denote the set of all the $n\times n$ real symmetric matrices. For any matrix $A\in\mathcal{S}^n$,  $A\succeq 0$ (or $\succ 0$) denotes that $A$ is positive semidefinite (or positive definite), and $\mbox{det}(A)$ denotes the determinant of $A$. For any matrix $B\in\mathcal{R}^{m\times n}$, $\rank(B)$, $\mbox{Null}(B)$ and  $\mbox{Range}(B)$ denote the rank, the null subspace and the range subspace of $B$, respectively. Moreover, $I$ denotes  the identity matrix of order $n$. Finally, for a smooth function $f:\mathcal{R}\rightarrow\mathcal{R}$, $f'$ denotes the first derivatives of $f$. 

\section{Basic properties of the local-nonglobal minimizers}
\label{sec-properties}

In this section, we shall study several basic properties that involve the local-nonglobal minimizers of the problems $(GTRE)$ and $(GTR)$. Some of them are well known but improved by us in their assumptions or proofs. And some of them are first presented by us.  These properties lay a strong foundation for our further discussions in the next section.

Firstly we state the classical optimality conditions for local minimizers of $(GTRE)$ and $(GTR)$, respectively.

\begin{lemma}[{\cite[Chapter 4]{bazaraa2006}}] \label{th:OptCond}
	{\bf (1) (The necessary optimality conditions for $(GTRE)$) } If $x^{*}$ is a local minimizer of $(GTRE)$ and $\nabla f_{1}(x^{*})\ne0$, then $f_{1}(x^{*})=0$ and there exists a Lagrangian multiplier $\mu^{*} \in \mathcal{R}$ such that
	\begin{subequations} 
		\begin{align}
			\label{eq:1NecCond}
			&\nabla f_{0}(x^{*})+\mu^{*}\nabla f_{1}(x^{*})=0,\\
			\label{eq:2NecCond}
			&v^{T}(A_{0}+\mu^{*}A_{1})v \geq 0,\ \forall\ v\in\mathcal{R}^{n}\,\mbox{ satisfying }\,\nabla f_{1}(x^{*})^{T}v=0.
		\end{align}
	\end{subequations}
	
	\noindent{\bf (2) (The sufficient optimality conditions for $(GTRE)$) } If $f_{1}(x^{*})$ $=0$ and there exists a Lagrangian multiplier $\mu^{*}\in\mathcal{R}$ satisfying \eqref{eq:1NecCond} and
	\begin{equation}\label{eq:2SufCond}
		v^{T}(A_{0}+\mu^{*}A_{1})v>0,\ \forall\ 0\neq v\in\mathcal{R}^{n} \mbox{ satisfying } \nabla f_{1}(x^{*})^{T}v=0,
	\end{equation}
	then $x^{*}$ is a strict local minimizer of $(GTRE)$.
	
	\noindent{\bf (3) (The necessary  optimality conditions for $(GTR)$) } If $x^{*}$ is a local minimizer of $(GTR)$ and $\nabla f_{1}(x^{*})\ne0$, then $f_{1}(x^{*})\leq0$ and there exists a Lagrangian multiplier $\mu^{*}\geq0$ such that \eqref{eq:1NecCond} and the complementary condition
	\begin{equation}\label{eq:2Complemen}
		\mu^{*}f_1(x^*)=0
	\end{equation}
	hold. Moreover,  if  $\mu^{*}=0$, then $x^{*}$ is a global minimizer of $(GTR)$; and if $\mu^{*}>0$ then \eqref{eq:2NecCond} holds. 
	
	\noindent{\bf (4) (The sufficient optimality conditions for $(GTR)$) }  If $f_{1}(x^{*})=0$ and there exists a Lagrangian multiplier $\mu^{*}>0$ satisfying \eqref{eq:1NecCond} and \eqref{eq:2SufCond}, then $x^{*}$ is a strict local minimizer of $(GTR)$. 
\end{lemma}

The following two lemmas were essentially contributed by Mor\'{e} \cite{more1993}, which tells us how to test the global minimizers of $(GTRE)$ and $(GTR)$ under some assumptions. 

\begin{lemma}[{\cite[Theorem 3.2]{more1993}}] \label{th:More1993-3.2}
	Suppose that $A_{1}\neq0$ and that $(GTRE)$ satisfies the two-side Slater's condition:
	\begin{equation}\label{eq:twosideslater}
		\exists\tilde{x}, \bar{x}\in\mathcal{R}^n \mbox{ such that } f_{1}(\tilde{x})<0<f_{1}(\bar{x}).
	\end{equation}
	A vector $x^{*}$ is a global minimizer of $(GTRE)$ if and only if $f_{1}(x^{*})=0$ and there exists a Lagrangian multiplier $\mu^{*}\in\mathcal{R}$ satisfying \eqref{eq:1NecCond} and $A_{0}+\mu^{*}A_{1}\succeq0.$
\end{lemma}

\begin{lemma}[{\cite[Theorem 1]{adachi2019}}] \label{th:More1993-3.4}
	Suppose that $(GTR)$ satisfies the Slater's condition:
	\begin{equation} \label{eq:salter-cons}
		\exists\tilde{x}\in\mathcal{R}^n \mbox{ such that } f_{1}(\tilde{x})<0.
	\end{equation}
	A vector $x^{*}$ is a global minimizer of $(GTR)$ if and only if $f_{1}(x^{*})\leq0$ and there exists a Lagrangian multiplier $\mu^{*}\geq0$ satisfying \eqref{eq:1NecCond} and 
	\begin{equation*}
		\mu^{*} f_{1}(x^{*})=0\,\, \mbox{ and } \,\, A_{0}+\mu^{*}A_{1}\succeq0.
	\end{equation*}
\end{lemma}

The following lemma tells us that any local-nonglobal minimizer of $(GTRE)$ or $(GTR)$ is a regular point. Its $(GTR)$ version was first given by Taati and Salahi \cite[Lemma 2.5]{taati2020}, and its $(GTRE)$ version was conditionally presented by Wang, Song and Xia \cite[Theorem 2.3]{xia2021}.  Here we present a combined proof for both $(GTR)$ and $(GTRE)$  more brief than ones in \cite{taati2020, xia2021}, and show that those preconditions on $(GTRE)$ given in \cite{xia2021} can be got rid off.

\begin{lemma}[{\cite[Lemma 2.5]{taati2020}; \cite[Theorem 2.3]{xia2021}}] \label{th:Nonglobal-LICQ}
	If $x^{*}$ is a local-nonglobal minimizer of $(GTRE)$ or $(GTR)$, then $f_{1}(x^{*})=0$ and $\nabla f_{1}(x^{*})\ne0$
	(i.e. $x^*$ is a regular point).
\end{lemma}

\begin{proof}
	It is apparent that $f_{1}(x^{*})=0$. As $x^{*}$ is a local-nonglobal minimizer, there is a feasible solution $\hat{x}\neq x^{*}$ of $(GTRE)$ (or $(GTR)$) such that
	\begin{equation}\label{eq:f0hatxsmaller}
		f_{0}(\hat{x})-f_{0}(x^*)<0.
	\end{equation} 
	Put $\hat{d}=\hat{x}-x^{*}$. Suppose by contradiction that $\nabla f_{1}(x^{*})=0$. Then one has
	\begin{equation*}%\label{eq:f1hatxTaylor}
		\hat{d}^{T}A_{1}\hat{d}=
		f_{1}(x^{*})+\nabla f_{1}(x^{*})^{T}\hat{d}+\hat{d}^{T}A_{1} \hat{d} =
		f_{1}(\hat{x})=0\,(\mbox{or}\leq0),
	\end{equation*} 
	which implies that
	\begin{equation*} %\label{eq:Sec2x*+tdfeasible}
		f_{1}(x^{*}+t \hat{d})=f_{1}(x^{*})+t \nabla f_{1}(x^{*})^{T}\hat{d}+t^{2}\hat{d}^{T}A_{1}\hat{d}=t^{2}\hat{d}^{T}A_{1}\hat{d}=0\,(\mbox{or}\leq0),\,\forall t\in\mathcal{R},
	\end{equation*}
	that is the straight line $\{x^*+t \hat{d}\,|\,t\in\mathcal{R}\}$ is feasible to  $(GTRE)$ (or $(GTR)$). So, due to $x^{*}$ is a local minimizer, there is a positive number $\delta>0$ such that 
	\begin{equation*}
		0\leq f_{0}(x^{*}+t \hat{d})-f_{0}(x^{*})=t \nabla f_{0}(x^{*})^{T}\hat{d}+t^2\hat{d}^{T}A_{0}\hat{d} ,\quad\forall t\in (-\delta,\delta),
	\end{equation*}
	which implies that $\nabla f_{0}(x^{*})^{T}\hat{d}=0$ and $\hat{d}^{T}A_{0}\hat{d}\geq0.$
	Taking $t=1$, one obtains that 
	\begin{equation}\label{eq:f0hatxgreater}
		f_{0}(\hat{x})-f_{0}(x^{*})=f_{0}(x^*+ \hat{d})-f_{0}(x^{*})=\hat{d}^{T}A_{0}\hat{d}\geq0,
	\end{equation}
	which contradicts with \eqref{eq:f0hatxsmaller}. Thus $\nabla f_{1}(x^{*})\ne0$ and the proof is completed.
\end{proof}

The following lemma shows that, if $(GTRE)$ has a local-nonglobal minimizer, then $f_1(x)$ is not a linear function.

\begin{lemma} \label{th:GTRE-A1ne0}
	If $x^{*}$ is a local-nonglobal minimizer of $(GTRE)$,  then $A_{1}\ne0$.
\end{lemma}

\begin{proof}
	As $x^{*}$ is  nonglobal, there exists another vector $\hat{x}\in\mathcal{R}^{n}$ satisfying $f_{1}(\hat{x})=0$ and $\hat{x}\ne x^*$ such that \eqref{eq:f0hatxsmaller} holds. Put $\hat{d}=\hat{x}-x^{*}$. Suppose by contradiction that $A_{1}=0$.  
	Then the straight line  $\{x^*+t \hat{d}\,|\,t\in\mathcal{R}\}$ connecting $x^{*}$ and $\hat{x}$ is feasible to  $(GTRE)$. Exactly following the proof of Lemma \ref{th:Nonglobal-LICQ}, one obtains the relation \eqref{eq:f0hatxgreater} also, which contradicts with \eqref{eq:f0hatxsmaller}. The proof is completed.
\end{proof}

\begin{remark}\label{rmk:s2-cexmpl-linear-GTR}
	The conclusion in Lemma \ref{th:GTRE-A1ne0} is not true for $(GTR)$, that is, $(GTR)$ may have a local-nonglobal minimizer even if $A_{1}=0$. For example, the problem 
	\begin{equation*} %\label{eq:s2conterexample-2}
		\min\{x_1^2-x_2^2\,\,\,|\,-x_1+x_2\leq0\}
	\end{equation*} 
	has a local-nonglobal minimizer $[1,1]^T$.
\end{remark}

The following lemma was essentially contributed by Mor\'{e} \cite{more1993}.

\begin{lemma}[{\cite[Lemma 3.1]{more1993}}] \label{th:More1993-3.1}
	If there is a vector $x^*\in\mathcal{R}^n$ such that $f_{1}(x^{*})=0$ and $\nabla f_{1}(x^{*})\neq0$, then the function $f_{1}$ is indefinite., that is, it satisfies the two-side Slater's condition \eqref{eq:twosideslater}.
\end{lemma} 

\begin{proof}
	As $\nabla f_{1}(x^{*})\neq0$, there is a positive number $\delta>0$ such that 
	\begin{equation*} %\label{eq:sec2deltaf1d+tda1d>0}
		h(t):=\Vert\nabla f_{1}(x^{*})\Vert^2+t\nabla f_{1}(x^{*})^TA_1\nabla f_{1}(x^{*})>0, \forall t\in[-\delta,\delta].
	\end{equation*}
	Take $\tilde{x}:=x^{*}-\delta\nabla f_{1}(x^{*})$ and $\bar{x}:=x^{*}+\delta\nabla f_{1}(x^{*})$. Note that $f_{1}(x^{*})=0$. Thus one obtains that
	\begin{equation*} 
		\begin{array}{ll}
			f_{1}(\tilde{x})=f_{1}(x^{*})-\delta\Vert\nabla f_{1}(x^{*})\Vert^2+\delta^2\nabla f_{1}(x^{*})^TA_1\nabla f_{1}(x^{*})=
			-\delta h(-\delta)<0 \mbox{ and} \vspace{1mm}\\
			f_{1}(\bar{x})=f_{1}(x^{*})+\delta\Vert\nabla f_{1}(x^{*})\Vert^2+\delta^2\nabla f_{1}(x^{*})^TA_1\nabla f_{1}(x^{*})=
			\delta h(\delta)>0.
		\end{array}
	\end{equation*}
	So $f_{1}(x)$ satisfies the two-side Slater's condition \eqref{eq:twosideslater}. 
\end{proof}

\begin{lemma} \label{th:multiplier-uniqueness}
	Suppose that the functions $f_0$, $f_1$ and a multiplier $\mu^{*}$ satisfy the first-order condition \eqref{eq:1NecCond} at $x^{*}$.  If $\nabla f_{1}(x^{*})\ne0$, then such the multiplier values $\mu^*$ that satisfy \eqref{eq:1NecCond} at $x^{*}$ are  unique. 
\end{lemma}

\begin{proof}
	Assume that there is another multiplier $\mu^{*}_1$  satisfying \eqref{eq:1NecCond}, that is,\\ $\nabla f_{0}(x^{*}) +\mu^*_1\nabla f_{1}(x^{*})=0$. Then one obtains that
	$$
	\left(\mu^{*}_1-\mu^{*}\right)\nabla f_{1}(x^{*})=0\,\,\Longrightarrow\,\,\mu^{*}_1=\mu^{*}, \mbox{ due to }\nabla f_{1}(x^{*})\ne0,
	$$
	which means that the corresponding Lagrangian multiplier values $\mu^*$ are unique.
\end{proof}

The following result was conditionally presented by Wang, Song and Xia \cite[Theorem 3.1]{xia2021}.
\begin{lemma}[{\cite[Theorem 3.1]{xia2021}}] \label{th:GTR-nonglobal-lambda>0}
	If $x^{*}$ is a local-nonglobal minimizer of  $(GTR)$ and $\mu^*\geq0$ is the Lagrangian multiplier  satisfying \eqref{eq:1NecCond} at $x^{*}$, then $\mu^*>0$.
\end{lemma}

\begin{proof}
	As $x^{*}$ is a local-nonglobal minimizer of  $(GTR)$, one has $f_{1}(x^{*})=0$ and $\nabla f_{1}(x^{*})\ne0$ from Lemma \ref{th:Nonglobal-LICQ}. Then the conclusion ``$\mu^*>0$" follows directly from Lemma \ref{th:OptCond}{(3)}. 
\end{proof}

\begin{remark} \label{eq:s2conterexample-1}
	Note that, for $(GTRE)$, the optimal Lagrangian multiplier of a local-nonglobal minimizer may be equal to zero. For example, the problem
	\begin{equation*} 
		\min\{x_1^2-2x_{1}-x_2^2\,\,\,|\,x_{1}x_{2}=0\}
	\end{equation*}
	has a unique local-nonglobal minimizer $x^{*}=[1,0]^{T}$ with the corresponding optimal Lagrangian multiplier $\mu^{*}=0$.
\end{remark}

The following lemma was first conditionally given by Taati and Salahi \cite[Lemma 3.1]{taati2020}.   As its proof can directly follow the proof of \cite[Lemma 3.1]{taati2020}, we omit the proof.

\begin{lemma}[{\cite[Lemma 3.1]{taati2020}}] \label{th:exac-one-negative-eigenvalue}
	If $x^{*}$ is a local-nonglobal minimizer of $(GTRE)$ or $(GTR)$ and  $\mu^{*}$ is the  Lagrangian multiplier  satisfying \eqref{eq:1NecCond} at $x^{*}$,  then the matrix $A_{0}+\mu^{*}A_{1}$ has exactly one negative eigenvalue. 
\end{lemma}

The following result plays an important role in this section. It shows that, for a local-nonglobal minimizer $x^*$ of $(GTRE)$ or $(GTR)$ and its associated Lagrangian multiplier $\mu^{*}$,  any vector $\bar{v}$ in the tangent space $\{\bar{v}\in\mathcal{R}^n\,|\,\nabla f_{1}(x^{*})^{T}\bar{v}=0\}$ satisfies either the sufficient condition \eqref{eq:2SufCond} or the following relation \eqref{eq:barvA1barv=barvA0barv=0}.

\begin{lemma} \label{th:barvTA1barv=barvTA0barv=0}
	If $x^{*}$ is a local-nonglobal minimizer of $(GTRE)$ or $(GTR)$,   then there exists a Lagrangian multiplier $\mu^{*}$ such that  \eqref{eq:1NecCond} and \eqref{eq:2NecCond} hold, and for any nonzero vector $\bar{v}\in\mathcal{R}^n$ that satisfies 
	\begin{equation}\label{eq:hatvTHhatv=0}
		\nabla f_{1}(x^{*})^{T}\bar{v}=0\mbox{ and } \bar{v}^{T}(A_{0}+\mu^{*}A_{1})\bar{v}=0,
	\end{equation}	
	there must be
	\begin{equation} \label{eq:barvA1barv=barvA0barv=0}
		\bar{v}^{T}A_{0}\bar{v}=\bar{v}^{T}A_{1}\bar{v}=0.
	\end{equation}
\end{lemma}

\begin{proof}  As $x^{*}$ is a local-nonglobal minimizer of $(GTRE)$ or $(GTR)$,  Lemmas \ref{th:Nonglobal-LICQ}, and \ref{th:OptCond}  guarantee that, there exists a Lagrangian multiplier $\mu^{*}$ such that $f_1(x^{*})=0$, $\nabla f_1(x^{*})\ne0$, \eqref{eq:1NecCond} and \eqref{eq:2NecCond} hold.  And from Lemma \ref{th:exac-one-negative-eigenvalue}, the Hessian matrix $A_{0}+\mu^{*}A_{1}$ has exactly one negative eigenvalue, say $\sigma_{1}<0$, with a corresponding unit eigenvector $v_{1}$  and $v_{1}^{T}(A_{0}+\mu^{*}A_{1})v_{1} =\sigma_{1}<0$,  which has apparently that 
	\begin{equation}\label{eq:f1v1}
		\nabla f_{1}(x^{*})^{T}v_{1}\neq0\quad \mbox{(due to \eqref{eq:2NecCond}).}
	\end{equation}
	
	Since $\bar{v}^{T}A_{1}\bar{v}=0$  implies  $\bar{v}^{T}A_{0}\bar{v}=0$ from \eqref{eq:hatvTHhatv=0}, we suppose by contradiction that the vector $\bar{v}$ taken by \eqref{eq:hatvTHhatv=0} satisfies
	\begin{equation}\label{eq:bar{v}A1bar{v}}
		\bar{v}^{T}A_{1}\bar{v}\neq 0.
	\end{equation}
	Note that \eqref{eq:hatvTHhatv=0} and \eqref{eq:f1v1} imply that the two vectors $v_1$ and $\bar{v}$ are linearly independent.
	Define $$x(\alpha,\beta):=x^{*}+\alpha v_{1}+\beta \bar{v},$$ where $\alpha$, $\beta\in\mathcal{R}$. Consider the following equation in $\alpha$ and $\beta$
	\begin{equation} \label{eq:1}
		\begin{array}{lll}
			0=f_{1}\left(x(\alpha,\beta)\right)\\
			=f_{1}(x^{*})+\nabla f_{1}(x^{*})^{T}(\alpha v_{1}+\beta \bar{v})+(\alpha v_{1}+\beta \bar{v})^{T}A_{1}(\alpha v_{1}+\beta \bar{v})\vspace{1mm}\\
			= \nabla f_{1}(x^{*})^{T}v_{1}\alpha+ v_{1}^{T}A_{1}v_{1}\alpha^{2}+ 2v_{1}^{T}A_{1}\bar{v}\alpha\beta+ \bar{v}^{T}A_{1}\bar{v}\beta^{2}\vspace{1mm}\\
			=\left(\bar{v}^{T}A_{1}\bar{v}\right)\left(-\gamma_{1}\alpha-\gamma_{2}\alpha^{2}-2\gamma_{3}\alpha\beta+\beta^{2}\right),\\
			\Longrightarrow\,\beta^{2}-2\gamma_{3}\alpha\beta-\gamma_{1}\alpha-\gamma_{2}\alpha^{2}=0.
		\end{array}
	\end{equation}
	where $\gamma_{1}=-\dfrac{\nabla f_{1}(x^{*})^{T}v_{1}}{\bar{v}^{T}A_{1}\bar{v}}$, $\gamma_{2}=-\dfrac{v_{1}^{T}A_{1}v_{1}}{\bar{v}^{T}A_{1}\bar{v}}$,  $\gamma_{3}=-\dfrac{v_{1}^{T}A_{1}\bar{v}}{\bar{v}^{T}A_{1}\bar{v}}$.
	From \eqref{eq:f1v1} and \eqref{eq:bar{v}A1bar{v}} there is $\gamma_{1}\neq0$.
	Without loss of generality, we assume $\gamma_{1}>0$. Then one can always find a positive number $\delta>0$ such that
	\begin{equation*}
		\gamma_{1}\alpha+\gamma_{2}\alpha^{2}=(\gamma_{1}+\gamma_{2}\alpha)\alpha>0, \,\,\,\,\forall \alpha\in (0,\delta\,],
	\end{equation*}
	which guarantees that the equation \eqref{eq:1} has two distinct real roots about $\beta$
	\begin{equation*}
		\beta=\gamma_{3}\alpha\pm\sqrt{\gamma_{3}^2\alpha^2+\gamma_{1}\alpha+\gamma_{2}\alpha^{2}}, \,\,\,\,\forall \alpha\in (0,\delta];
	\end{equation*}
	one root is positive and the other is negative. So we define the function $\beta(\alpha)$
	$$
	\beta(\alpha):=
	\begin{cases}
		\gamma_{3}\alpha+\sqrt{\gamma_{3}^2\alpha^2+\gamma_{1}\alpha+\gamma_{2}\alpha^{2}} \,\,\mbox{ as }v_1^T\bar{v}\geq0,\vspace{1mm}\\
		\gamma_{3}\alpha-\sqrt{\gamma_{3}^2\alpha^2+\gamma_{1}\alpha+\gamma_{2}\alpha^{2}} \,\,\mbox{ as }v_1^T\bar{v}<0,
	\end{cases}
	\forall \alpha\in [0,\delta],
	$$
	such that 
	\begin{equation*}
		%\label{eq:beta(alpha)v1hatvgeq0}
		\beta(\alpha)v_1^T\bar{v}\geq0,\,\,\forall \alpha\in [0,\delta].
	\end{equation*}
	Therefore, one obtains the following relations: 
	\begin{equation*} %\label{eq:f0-f0}
		\begin{array}{lll}
			x\left(0,\beta(0)\right)=x^*\mbox{ and } x\left(\alpha,\beta(\alpha)\right)\ne x^*,\,\forall \alpha\in (0,\delta];\vspace{1mm}\\
			x\left(\alpha,\beta(\alpha)\right)\longrightarrow x^*\,\,\,(\mbox{ as } \alpha \longrightarrow 0^+);\vspace{1mm}\\
			f_{1}\left(x\left(\alpha,\beta(\alpha)\right)\right)\equiv0,\,\,\forall \alpha\in [0,\delta];\vspace{1mm}\\
			f_{0}\left(x\left(\alpha,\beta(\alpha)\right)\right)-f_{0}(x^{*})\vspace{1mm}\\
			= f_{0}\left(x\left(\alpha,\beta(\alpha)\right)\right)+\mu^{*}f_{1}\left(x\left(\alpha,\beta(\alpha)\right)\right) -\left( f_{0}(x^{*})+\mu^{*}f_{1}(x^{*})\right)\vspace{1mm}\\ 
			=\left(\alpha v_{1}+\beta(\alpha)\bar{v}\right)^{T}(A_{0} +\mu^{*}A_{1})\left(\alpha v_{1}+\beta(\alpha)\bar{v}\right) \vspace{1mm}\\
			=\sigma_{1}\left(\alpha^{2}+2\alpha\beta(\alpha)v_1^T\bar{v}\right)<0,\,\,\forall \alpha\in (0,\delta];
		\end{array}
	\end{equation*}
	which contradicts with $x^*$ being a local-nonglobal minimizer. Thus the proof is completed.
\end{proof}

\begin{remark}\label{rmk:s2-cexmpl-0} 
	Such the nonzero vector $\bar{v}$ as shown in the above lemma may indeed appear in practice.
	Consider the following problem
	\begin{equation*} %\label{eq:s2example-3}
		\min\{x_{1}x_{2}-x_{2}^{2}\,\,|\,\,-x_{1}x_{2}=(\mbox{or}\leq)\, 0\}. 
	\end{equation*}
	It has local-nonglobal minimizers $x^{*}=[\alpha,0]^{T}$ for all $\alpha\ne0$ with $f_1(x^*)=0$ and $\nabla f_1(x^*)=[0,-\alpha]^T$. These local-nonglobal minimizers have the same optimal Lagrangian multiplier $\mu^{*}=1$. The corresponding Lagrangian Hessian matrix is  
	\begin{equation*}
		A_{0}+\mu^{*}A_{1}=
		\begin{bmatrix}
			0 & \frac{1}{2}\\
			\frac{1}{2} & -1
		\end{bmatrix}
		+
		\begin{bmatrix}
			0 & -\frac{1}{2}\\
			-\frac{1}{2} & 0
		\end{bmatrix}
		=\begin{bmatrix}
			0 & 0\\
			0 & -1
		\end{bmatrix}, 
	\end{equation*}
	which has a zero eigenvalue with a unit eigenvector $\bar{v}=[1,0]^T$.  And $\bar{v}$ satisfies
	$$
	\nabla f_1(x^*)^T\bar{v}=\bar{v}^TA_0\bar{v}=\bar{v}^TA_1\bar{v}=0.
	$$
	This example also shows that, if $f_0$ and $f_1$ are both homogeneous but $c_1=0$, $(GTRE)$ and $(GTR)$ may still have local-nonglobal minimizers.
\end{remark}

\begin{remark}\label{rmk:s2-cexmpl-1}
	The converse proposition of Lemma \ref{th:barvTA1barv=barvTA0barv=0} is not true. Consider a counterexample as follows: 
	\begin{equation} \label{eq:s2example-4}
		\min\{f_0(x)=2x_{1}x_{2}-2x_{2}x_{3}+x_{3}^{2}-2x_{1}\,\,|\,\,f_1(x)=2x_{2}x_{3}+2x_{1}=(\mbox{or}\leq)\, 0\}. 
	\end{equation}
	Take $x^*=[0,0,0]^T$ and $\mu^*=1$. One can easily verify that the point $x^*$ and the multiplier $\mu^*$ satisfy $f_1(x^*)=0$ and $\nabla f_1(x^*)=[2,0,0]^T\ne0$, the necessary optimality conditions \eqref{eq:1NecCond} and \eqref{eq:2NecCond}, where
	\begin{equation*}
		A_{0}+\mu^{*}A_{1}=
		\begin{bmatrix}
			0 &1&0\\
			1&0 &-1\\
			0&-1&1
		\end{bmatrix}
		+
		\begin{bmatrix}
			0 &0&0\\
			0&0 &1\\
			0&1&0
		\end{bmatrix}
		=\begin{bmatrix}
			0 &1&0\\
			1&0 &0\\
			0&0&1
		\end{bmatrix}.
	\end{equation*} 
	And all the vectors that satisfy \eqref{eq:hatvTHhatv=0} are just $\bar{v}=[0,\alpha,0]^T$ $(\forall \alpha\in\mathcal{R})$, which certainly satisfy
	\eqref{eq:barvA1barv=barvA0barv=0} too. However, $x^*$ is not a local minimizer of \eqref{eq:s2example-4} yet, because the point pencil  $x(t)=[-t^3,\,t,\,t^2\,]^T$ $(t\in\mathcal{R})$ satisfy $f_1(x(t))=0$, $x(0)=x^*$ and $f_0(x(t))=-t^4<0=f_0(x^*)$  $\forall t\ne0$. 
\end{remark} 

\begin{remark}\label{rmk:s2-cexmpl-2}
	By slightly redeveloping the problem \eqref{eq:s2example-4}, one can also verify that a conjecture presented by Xia's team \cite[Conjecture 5.10]{xia2023} is not true. Observe the following example.
	\begin{equation}\label{eg:s2conj_eg}
		\begin{array}{lll}
			&\min & f_0(x)=x^TA_0x=2x_{1}x_{2}-2x_{2}x_{3}+x_{3}^{2}-2x_{1}x_4\\
			&\mbox{\rm s.t.} & f_1(x)=x^TA_1x-1=x_1^2+x_2^2+x_3^2+x_4^2-1=0,\\
			& & f_2(x)=x^TA_2x-1=2x_2x_3+2x_1x_4+x_1^2+x_2^2+x_3^2+x_4^2-1\leq0.
		\end{array}
	\end{equation}
	Firstly, the problem \eqref{eg:s2conj_eg} satisfies the Slater's condition and the joint definiteness condition i.e. $\exists\, \mu_0,\, \mu_1,\, \mu_2\in\mathcal{R}\,\, \mbox{s.t.}\,\, \mu_0A_0+\mu_1A_1+\mu_2A_2\succ0$. Moreover, the point $x^*=[0,0,0,1]^T$ and the multiplier pair $(\mu^*_1,\mu^*_2)=(-1,1)$ satisfy $f_1(x^*)=0=f_2(x^*)$ and the necessary optimality conditions
	\begin{equation*}
		\begin{array}{lll}
			(A_0+\mu^*_1 A_1+\mu^*_2 A_2)x^*=0
			\mbox{ and}\\
			v^T(A_0+\mu^*_1 A_1+\mu^*_2 A_2)v\geq0,\forall v\in\mathcal{R}^4 
			\mbox{ satisfying }
			v^TA_1x^*=0=v^TA_2x^*, \mbox{ where} \vspace{1mm}\\
			A_0+\mu^*_1 A_1+\mu^*_2 A_2=
			\begin{bmatrix}
				0 &1&0&0\\
				1&0 &0&0\\
				0&0&1&0\\
				0&0&0&0
			\end{bmatrix}.
		\end{array}
	\end{equation*}
	And there is a nonzero vector $\bar{v}=[0,1,0,0]^T$ such that $\bar{v}^TA_1x^*=0=\bar{v}^TA_2x^*$ and $\bar{v}^T(A_0+\mu^*_1 A_1+\mu^*_2 A_2)\bar{v}=0=\bar{v}^T(A_2-A_1)\bar{v}$.	
	However, $x^*$ is not a local minimizer of \eqref{eg:s2conj_eg} yet, because  the point pencil
	\begin{equation*}
		x(t)=\dfrac{1}{\sqrt{t^6+t^2+t^4+1}}\,[-t^3,t,t^2,1]^T,\,\,\forall t\in\mathcal{R}
	\end{equation*}
	satisfy $f_1(x(t))=0=f_2(x(t))$, $x(0)=x^*$ and $f_0(x(t))=-t^4/({t^6+t^2+t^4+1})<0=f_0(x^*)$ $\forall t\neq0$.
\end{remark}

By using Lemma \ref{th:barvTA1barv=barvTA0barv=0}, we obtain an important property involving the strict local-nonglobal minimizers of $(GTRE)$ and $(GTR)$ under no assumptions.

\begin{lemma} \label{th:strict-nonglobal-suf-cond}
	If $x^{*}$ is a strict local-nonglobal minimizer of $(GTRE)$ or $(GTR)$, then there is  a Lagrangian multiplier $\mu^{*}$ such that  the sufficient optimality conditions \eqref{eq:1NecCond} and \eqref{eq:2SufCond} hold.
\end{lemma}

\begin{proof}
	As $x^{*}$ is a strict local-nonglobal minimizer of $(GTRE)$ or $(GTR)$, it holds that $f_1(x^*)=0$ and $\nabla f_1(x^*)\ne0$, and there exists a unique Lagrangian multiplier $\mu^{*}$ such that the necessary optimality conditions  \eqref{eq:1NecCond} and \eqref{eq:2NecCond} hold, due to Lemma \ref{th:OptCond}, Lemma \ref{th:Nonglobal-LICQ} and Lemma \ref{th:multiplier-uniqueness}. We suppose by contradiction that \eqref{eq:2SufCond} is violated, that is, there is a nonzero vector $\bar{v}\in\mathcal{R}^n$ that satisfies \eqref{eq:hatvTHhatv=0}. From Lemma \ref{th:barvTA1barv=barvTA0barv=0}, $\bar{v}$ also satisfies \eqref{eq:barvA1barv=barvA0barv=0}. Then for all $\alpha\in\mathcal{R}$, one has 
	\begin{equation*}
		\begin{cases}
			f_1(x^*+\alpha\bar{v})=f_1(x^*)+ \alpha(\nabla f_1(x^*))^T\bar{v}+\alpha^2\bar{v}^TA_1\bar{v} =0+0+0=0,\\
			f_0(x^*+\alpha\bar{v})=f_0(x^*)- \alpha\mu^*(\nabla f_1(x^*))^T\bar{v}+\alpha^2\bar{v}^TA_0\bar{v} =f_0(x^*)-0+0=f_0(x^*),
		\end{cases}
	\end{equation*}
	which contradicts with the assumption that $x^*$ is a strict local-nonglobal minimizer.  
\end{proof}

By the above lemma, one can reveal another interesting property about the strict local-nonglobal minimizers of $(GTRE)$ and $(GTR)$, which is stated as follows. 

\begin{lemma} \label{th:no-zero-eigenvalues}
	If $x^{*}$ is a strict local-nonglobal minimizer of $(GTRE)$ or $(GTR)$ and  $\mu^{*}$ is the corresponding Lagrangian multiplier to $x^{*}$, then the matrix  $A_{0}+\mu^{*}A_{1}$ is nonsingular.
\end{lemma}

\begin{proof} 	
	From Lemma \ref{th:exac-one-negative-eigenvalue}, the matrix $A_{0}+\mu^{*}A_{1}$ has exactly one  negative eigenvalue, say $\sigma_{1}<0$, with a unit eigenvector $v_{1}$. Suppose on the contrary that $A_{0}+\mu^{*}A_{1}$ has a zero eigenvalue, say $\sigma_{2}=0$ with one other unit eigenvector $v_2$. It is apparent that $v_2$ are linearly independent of $v_1$. Furthermore, from  Lemma \ref{th:strict-nonglobal-suf-cond}, one yields
	\begin{equation*}
		\nabla f_{1}(x^{*})^{T}v_2\neq0.
	\end{equation*}
	Define a vector
	\begin{equation*}
		\hat{v}:=v_1-\dfrac{\nabla f_{1}(x^{*})^{T}v_1}{\nabla f_{1}(x^{*})^{T}v_2}\,v_2\ne0.
	\end{equation*}
	Then one obtains that
	\begin{equation*}
		\hat{v}^T\left(A_{0}+ \mu^{*}A_{1} \right)\hat{v}=\sigma_1<0,\quad 
		\nabla f_{1}(x^{*})^{T}\hat{v}=0  \mbox{ and } \hat{v}\ne0,
	\end{equation*}
	which is a contradiction.
\end{proof} 

Now we are ready to present  two key theorems of this section, which characterize whether or not a point $x^*$ is a strict local-nonglobal minimizer of $(GTRE)$ (or $(GTR)$) under no assumptions.

\begin{theorem}\label{th:GTRE-strict-nonglobal-sufnec}
	A vector  $x^{*}$ is a strict local-nonglobal minimizer of $(GTRE)$, if and only if the following three conditions hold: {\bf (i)} $f_{1}(x^{*})=0$; {\bf (ii)}
	there exists a Lagrangian multiplier  $\mu^{*}\in\mathcal{R}$ such that the sufficient optimality conditions \eqref{eq:1NecCond}  and \eqref{eq:2SufCond} hold with $A_{0}+\mu^{*}A_{1}\not\succeq0$; {\bf (iii)}  $A_1\ne0$.
\end{theorem}

\begin{proof}
	{\bf ``$\Longrightarrow$". } As  $x^{*}$ is a strict local-nonglobal minimizer of $(GTRE)$, the three conditions follow directly from Lemmas \ref{th:Nonglobal-LICQ}, \ref{th:OptCond}, \ref{th:exac-one-negative-eigenvalue}, \ref{th:strict-nonglobal-suf-cond} and  \ref{th:GTRE-A1ne0}.
	
	{\bf ``$\Longleftarrow$". } Assume that the three conditions (i), (ii) and (iii) hold. According to Lemma \ref{th:OptCond}(2), the conditions (i) and (ii) guarantee that $x^*$ is a strict local minimizer of $(GTRE)$. So one needs only to verify that $x^{*}$ is not a global minimizer.  In fact, the condition (ii) ensures $\nabla f_1(x^*)\ne0$ because otherwise then \eqref{eq:2SufCond} implies $A_{0}+\mu^{*}A_{1}\succ0$, which contradicts with $A_{0}+\mu^{*}A_{1}\not\succeq0$. Therefore, the function $f_1$ satisfies the two-side Slater's condition and the current $\mu^*$ is the unique optimal Lagrangian multiplier corresponding to $x^{*}$,  due to Lemmas \ref{th:More1993-3.1} and  \ref{th:multiplier-uniqueness}. Then, from  the condition (iii), $A_{0}+\mu^{*}A_{1}\not\succeq0$ and Lemma \ref{th:More1993-3.2}, one obtains that $x^*$ is not a global minimizer. The proof is completed.
\end{proof}

\begin{theorem}\label{th:GTR-strict-nonglobal-sufnec}
	A vector  $x^{*}$ is a strict local-nonglobal minimizer of $(GTR)$, if and only if the following two conditions hold: {\bf (i)} $f_{1}(x^{*})=0$; {\bf (ii)} there exists a Lagrangian multiplier $\mu^{*}>0$, such that the sufficient optimality conditions \eqref{eq:1NecCond} and \eqref{eq:2SufCond} hold with  $A_{0}+\mu^{*}A_{1}\not\succeq0$. 
\end{theorem}

\begin{proof}
	The proof process can directly follow that of Theorem \ref{th:GTRE-strict-nonglobal-sufnec}, except Lemmas \ref{th:More1993-3.4} and \ref{th:GTR-nonglobal-lambda>0} need to be  utilized additionally. So we omit it.
\end{proof}

Now we introduce the concept of the joint definiteness of two matrices.
\begin{definition} \label{df:Joint-Definite}
	We say that two symmetric matrices $A_0$ and $A_1$ are jointly definite if $\exists\mu_1\in\mathcal{R}$ such that 
	\begin{equation} \label{eq:joint-definite}
		\begin{array}{cll}
			\mbox{ either }& A_0+\mu_1A_1\succ0& (\mbox{jointly positive definite})\vspace{1mm}\\
			\mbox{ or } &A_0+\mu_1A_1\prec0& (\mbox{jointly negative definite}).
		\end{array}
	\end{equation}  
\end{definition}
Note that Definition \ref{df:Joint-Definite} can be equivalently stated as follows: we say that $A_0$ and $A_1$ are jointly definite if 
\begin{equation} \label{eq:joint-definite-1}
	\exists\,\mu_{0}, \mu_{1}\in\mathcal{R}\,\, \mbox{ such that }\,\,\mu_{0}A_{0}+\mu_{1}A_{1} \succ 0.
\end{equation}

The following theorem shows that, under only the assumption that $A_0$ and $A_1$ are jointly definite, all the local-nonglobal minimizers of $(GTRE)$ and $(GTR)$ must be strict, and so they all satisfy the sufficient optimality conditions \eqref{eq:1NecCond}  and \eqref{eq:2SufCond} according to Theorems  \ref{th:GTRE-strict-nonglobal-sufnec} and \ref{th:GTR-strict-nonglobal-sufnec}. Similar results  (\cite[Theorem 4.8]{xia2023}, \cite[Theorems 3.2 and 3.3]{xia2021}) were presented by Xia's team under stricter assumptions than ours.

\begin{theorem} \label{th:nonstrict-nonglobal-zero}
	Suppose that $A_0$ and $A_1$ are jointly definite. Then $(GTRE)$ and $(GTR)$ have no nonstrict local-nonglobal minimizers.
\end{theorem}

\begin{proof}
	As $A_0$ and $A_1$ are jointly definite, the formula \eqref{eq:joint-definite-1} holds. We suppose by contradiction that $(GTRE)$ (or $(GTR)$) has one nonstrict local-nonglobal minimizer $x^*$. By using Lemmas \ref{th:OptCond} and \ref{th:barvTA1barv=barvTA0barv=0}, there exists a Lagrangian multiplier $\mu^*$ and a nonzero vector $\bar{v}\ne0$ such that $\left(\nabla f_0(x^*)+\mu^*f_1(x^*)\right)^T\bar{v}=\bar{v}^T(A_0+\mu^*A_1)\bar{v}=\bar{v}^TA_0\bar{v}=\bar{v}^TA_1\bar{v}=0$. It implies that $\bar{v}^T(\mu_{0}A_{0}+\mu_{1}A_{1})\bar{v}=0$, which contradicts with the joint definiteness assumption \eqref{eq:joint-definite-1}.
\end{proof}

The following theorem establishes a connection between the local-nonglobal minimizers of $(GTRE)$ and $(GTR)$, which will play a crucial role in the remaining discussions about the number of the local-nonglobal minimizers of $(GTR)$ in the next section.

\begin{theorem}\label{th:ineq-nonglobal=eq-nonglobal}
	Suppose that $A_0$ and $A_1$ are jointly definite, and suppose that $n\geq2$. A vector $x^{*}$ is a local-nonglobal minimizer of  $(GTR)$ if and only if $x^{*}$ is a local-nonglobal minimizer of the corresponding  $(GTRE)$ with $\mu^*>0$.
\end{theorem}

\begin{proof}
	{\bf ``$\Longleftarrow$". } It follows directly from Theorems \ref{th:nonstrict-nonglobal-zero}, \ref{th:GTRE-strict-nonglobal-sufnec} and \ref{th:GTR-strict-nonglobal-sufnec}.
	
	{\bf ``$\Longrightarrow$". }  Assume that $x^{*}$ is a local-nonglobal minimizer of  $(GTR)$.  By Theorems \ref{th:nonstrict-nonglobal-zero} and \ref{th:GTR-strict-nonglobal-sufnec}, there exists a Lagrangian multiplier $\mu^{*}>0$ such that $f_1(x^*)=0$ and the sufficient optimality conditions \eqref{eq:1NecCond} and \eqref{eq:2SufCond} hold with  $A_{0}+\mu^{*}A_{1}\not\succeq0$. From Theorem  \ref{th:GTRE-strict-nonglobal-sufnec}, one needs only to verify $A_1\ne0$ to complete the proof. We suppose by contradiction that $A_1=0$. It yields that
	\begin{equation} \label{eq:s3A0<0}
		\begin{cases}
			\mu_{0}A_0+\mu_1A_{1}\succ0\,(\because\,\eqref{eq:joint-definite-1})\Longrightarrow\, \mu_{0}A_0\succ0\\
			A_{0}+\mu^{*}A_{1}\not\succeq0\,\Longrightarrow\, A_0\not\succeq0
		\end{cases}
		\,\Longrightarrow\, A_0\prec0.
	\end{equation}
	\eqref{eq:s3A0<0} means from the assumption ``$n\geq2$" that $A_0=A_{0}+\mu^{*}A_{1}$ has at least two negative eigenvalues, which contradicts with Lemma \ref{th:exac-one-negative-eigenvalue}.
\end{proof}

\begin{remark} \label{rmk:s2-cexmpl-3}
	The assumption ``$n\geq2$" is necessary in the above theorem, which is illustrated by the following counterexample. The  $(GTR)$ problem 
	\begin{equation}\label{eq:example-n-greater2}
		\min\{x(1-x)\,|\,x-1\leq0\}
	\end{equation}
	satisfies \eqref{eq:joint-definite-1} and has a strict local-nonglobal minimizer $x^*=1$. However, $x^*=1$ is a global minimizer of the corresponding $(GTRE)$ problem: $\min\{x(1-x)\,|\,x-1=0\}$. 
\end{remark}

The proof process of the above theorem reveals a fact as follows.
\begin{corollary}
	Suppose that $A_0$ and $A_1$ are jointly definite, and suppose that $n\geq2$. If $(GTR)$ has a local-nonglobal minimizer, then $A_1\ne0$.
\end{corollary}

\section{The number of the local-nonglobal minimizers} \label{sec-number}
In this section, we shall discuss how many local-nonglobal minimizers the problem $(GTRE)$ or $(GTR)$  may have under the joint definiteness condition \eqref{eq:joint-definite}.  On this issue, by Theorem \ref{th:ineq-nonglobal=eq-nonglobal}, the problem $(GTR)$ may come down to the problem $(GTRE)$ except for the trivial case ``$n=1$". So we need only to consider the problem $(GTRE)$. 
And the problem $(GTRE)$ defined in \eqref{eq:equalityproblem} is exactly as same as the following problem:
\begin{equation}\label{eq:s3M1}
	\begin{array}{lcl}
		&\underset{x\in\mathcal{R}^n}{\min} &f_{0}(x)+\mu_{1}f_{1}(x)=x^{T}(A_{0}+\mu_{1}A_{1})x+2(b_{0}+\mu_{1}b_{1})x+\mu_{1}c_1\\
		&\mbox{s.t.}& f_{1}(x)=x^{T}A_{1}x+2b_{1}^{T}x+c_{1}=0,
	\end{array}
\end{equation}	
where $\mu_1$ is given by \eqref{eq:joint-definite} such that $A_0+\mu_1A_1\succ0$  or $A_0+\mu_1A_1\prec0$. Perform a Cholesky factorization on $A_0+\mu_1A_1$ or $-(A_0+\mu_1A_1)$ such that $A_0+\mu_1A_1=\pm LL^T$, where $L$ is a lower triangular matrix with positive diagonal elements. Then one finds an invertible linear transformation with a parallel translation
\begin{equation}\label{eq:s3-Qq-trans}
	x=Q^Ty+q
\end{equation}
such that
$
f_{0}(Q^Ty+q)+\mu_{1}f_{1}(Q^Ty+q)=\pm\Vert y \Vert^{2}+c_0,
$
where $Q=L^{-1}$, $q=\mp Q^TQ(b_0+\mu_1b_1)$ and $c_0=(b_0+\mu_1b_1)^Tq+\mu_{1}c_1$. Compute 
$$
g(y):=f_{1}(Q^Ty+q)=y^{T}{A} y+2{b}^{T}y+{c},
$$
where ${A}=QA_1Q^T$,  ${b}=QA_1q+Qb_1$ and ${c}=q^TA_1q+2b_1^Tq+c_1$. Then the problem \eqref{eq:s3M1} is equivalently converted into the problem \eqref{eq:s3M2}  
\begin{equation}\label{eq:s3M2}
	\begin{array}{lcl}
		&\underset{y\in\mathcal{R}^n}{\min} &\Vert y \Vert^{2}\\
		&\mbox{s.t.} &g(y)=y^{T}{A} y+2{b}^{T}y+{c}=0,
	\end{array}
\end{equation}	
or the problem \eqref{eq:s3M3}
\begin{equation}\label{eq:s3M3}
	\begin{array}{lcl}
		&\underset{y\in\mathcal{R}^n}{\min} &-\Vert y \Vert^{2}\\
		&\mbox{s.t.} &g(y)=y^{T}{A} y+2{b}^{T}y+{c}=0.
	\end{array}
\end{equation}	 
In this section, we always denote the eigenvalue decomposition of ${A}$ by 
\begin{equation} \label{eq:s3-A-eigen-notation}
	{A}=V\Lambda V^T,\quad \Lambda=\diag\,(\lambda_1,\lambda_2,\cdots,\lambda_n),\quad V=[v_1,v_2,\cdots,v_n],
\end{equation}
where  $\lambda_1\leq\lambda_2\leq\cdots\leq\lambda_n$ and $V$ is an orthogonal matrix. Define
\begin{equation} \label{eq:s3-ri-def}
	r_i:=v_i^T{b},\quad i=1,2,\cdots,n.
\end{equation}

\subsection{On the problem \texorpdfstring{\eqref{eq:s3M2}}{}}

We first discuss  the problem \eqref{eq:s3M2}. For this problem, the following lemma describes the range of the corresponding Lagrangian multiplier to a local-nonglobal minimizer.

\begin{lemma}\label{th:s3mu*range}
	If $y^{*}$ is a local-nonglobal minimizer of the problem \eqref{eq:s3M2}, then $g(y^*)$ $=0$ and there is a  unique Lagrangian multiplier $0\ne\mu^{*}\in\mathcal{R}$ such that
	\begin{equation} \label{eq:s3-first-nec}
		\left(I+\mu^*{A}\right)y^*+\mu^{*}{b} =0.
	\end{equation}  
	Moreover, if $\mu^{*}>0$ then   $\max\{-\lambda_2,0\}<\dfrac{1}{|\mu^*|}<-\lambda_1$ and  $r_{1}\neq0$; else if $\mu^{*}<0$ then  $\max\{\lambda_{n-1},0\}<\dfrac{1}{|\mu^*|}<\lambda_n$ and $r_{n}\neq0$.
\end{lemma}

\begin{proof}
	As $y^{*}$ is a local-nonglobal minimizer of the problem \eqref{eq:s3M2}, by applying Theorems \ref{th:GTRE-strict-nonglobal-sufnec} and \ref{th:nonstrict-nonglobal-zero} to the problem \eqref{eq:s3M2}, one has $g(y^*)=0$ and finds a unique Lagrangian multiplier $\mu^{*}\in\mathcal{R}$ such that  the sufficient optimality conditions \eqref{eq:1NecCond} and \eqref{eq:2SufCond} hold with  $I+\mu^{*}A\nsucceq0$, which implies that \eqref{eq:s3-first-nec} holds and $\mu^*\ne0$.  Furthermore, by Lemma \ref{th:exac-one-negative-eigenvalue} and Lemma \ref{th:no-zero-eigenvalues}, $I+\mu^{*}{A}$ is a nonsingular matrix with exactly one negative eigenvalue, which means that  
	\begin{equation} \label{s3l1mu*l2}
		\begin{array}{ll}
			\mbox{if } \mu^{*}>0\Longrightarrow 1+\mu^*\lambda_1<0<1+\mu^*\lambda_2 \,\,\,\, \Longrightarrow\max\{-\lambda_2,0\}< \dfrac{1}{|\mu^*|}<-\lambda_1;\vspace{1mm}\\
			\mbox{if } \mu^{*}<0\Longrightarrow 1+\mu^*\lambda_n<0<1+\mu^*\lambda_{n-1}\Longrightarrow\max\{\lambda_{n-1},0\}<\dfrac{1}{|\mu^*|}<\lambda_n.
		\end{array}
	\end{equation}
	So $v_i^T\left(I+\mu^*{A}\right)v_i =1+\mu^*\lambda_i$ ($i=1,n$), \eqref{s3l1mu*l2} and the necessary optimality condition \eqref{eq:2NecCond} guarantees that
	\begin{equation}\label{eq:s3-v1ndeltagne0}
		\begin{array}{ll}
			\mbox{ if } \mu^{*}>0\,\Longrightarrow\, 0\ne v_1^T\nabla g(y^*)=2v_1^T(Ay^*+b)=2(\lambda_1v_1^Ty^*+r_1), \vspace{2mm}\\
			\mbox{ if } \mu^{*}<0\,\Longrightarrow\,  0\ne v_n^T\nabla g(y^*)=2v_n^T(Ay^*+b)=2(\lambda_nv_n^Ty^*+r_n),
		\end{array}
	\end{equation} 
	where the vectors $v_i$ ($i=1,n$) are unit eigenvectors of the matrix $A$ defined in \eqref{eq:s3-A-eigen-notation}.
	Note that from \eqref{eq:s3-first-nec} one obtains that  
	\begin{equation*}
		\mu^*r_i=\mu^*v_i^Tb=-v_i^T\left(I+\mu^*{A}\right)y^*=-(1+\mu^*\lambda_i)v_i^Ty^*, \quad (i=1,n),
	\end{equation*}
	which deduces that 
	\begin{equation}\label{eq:s3-v1ny-innp}
		v_1^Ty^*=\dfrac{-\mu^*r_1}{1+\mu^*\lambda_1}
		\mbox{ and }
		v_n^Ty^*=\dfrac{-\mu^*r_n}{1+\mu^*\lambda_n},
	\end{equation}
	where $(1+\mu^*\lambda_1)(1+\mu^*\lambda_n)\ne0$ because of the nonsingularity of $I+\mu^{*}{A}$. Substituting
	\eqref{eq:s3-v1ny-innp} into \eqref{eq:s3-v1ndeltagne0}  yields that
	\begin{equation*}
		\begin{cases}
			\mbox{ if } 
			\mu^{*}>0,\,0\ne v_1^T\nabla g(y^*) =\dfrac{2r_1}{1+\mu^*\lambda_1} \,\Longrightarrow\,r_1\ne0,\vspace{1mm}\\
			\mbox{ if } \mu^{*}<0,\,0\ne v_n^T\nabla g(y^*) =\dfrac{2r_n}{1+\mu^*\lambda_n} \, \Longrightarrow \,r_n\ne0.
		\end{cases}
	\end{equation*}
	The proof is completed.
\end{proof}

Note that the case ``$\mu^*<0$" can exactly regarded as the situation ``$-\mu^*>0$ with the equality constraint $-g(y^*)=0$". Thus, from now on, we focus on discussing the case ``$\mu^*>0$". Put $\eta^*=1/{\mu^*}$. One knows from Lemma \ref{th:s3mu*range} that a local-nonglobal minimizer $y^*$  with $\eta^*>0$  for the problem \eqref{eq:s3M2}  must satisfy 
\begin{equation}\label{eq:y*eta*cond}
	\max\{-\lambda_2,0\}<\eta^*<-\lambda_1,\,r_{1}\neq0,\, y^*=-(A+\eta^*I)^{-1}b\, \mbox{ and }\, g(y^*)=0.
\end{equation}
Since the matrix $A+\eta I$ is nonsingular for all $\eta\in\left(-\lambda_2,-\lambda_1\right)$, we define a vector function $y_1(\eta)$ and a function $\psi_1(\eta)$ by
\begin{equation}\label{eq:y-psi(eta)}
	\begin{cases}
		y_1(\eta):=-(A+\eta I)^{-1}b=-\sum\limits_{i=1}^{n}\dfrac{r_{i}}{\lambda_{i}+\eta}\,v_i,\vspace{1mm}\\
		\psi_1(\eta):=g(y_1(\eta))=\sum\limits_{i=1}^{n} \dfrac{\lambda_ir_{i}^{2}}{(\lambda_{i}+\eta)^{2}}-\sum\limits_{i=1}^{n} \dfrac{2r_{i}^{2}}{\lambda_{i}+\eta}+c,
	\end{cases}
	\,\forall\eta\in\left(-\lambda_2,-\lambda_1\right),
\end{equation}
where $\max\{-\lambda_2,0\}<-\lambda_1$ and $r_1\ne0$.  Note that $\eta^*$ in \eqref{eq:y*eta*cond} is a zero point of $\psi_1$. A natural question is: how many zero points may the function $\psi_1$  have on the interval $(\max\{-\lambda_2,0\},\,-\lambda_1)$? To answer this question, we list several useful properties about $y_1(\eta)$ and $\psi_1(\eta)$: 
\begin{equation}\label{eq:s3nablag-psi'}
	\begin{cases}
		y_1(\eta)\ne0,\quad \forall \eta\in \left(-\lambda_2,-\lambda_1\right);\vspace{1mm}\\
		\nabla g|_{y_1(\eta)}=2Ay_1(\eta)+2b=-2\eta y_1(\eta)\ne0,\quad\forall \eta\in \left(\max\{-\lambda_2,0\},-\lambda_1\right);\vspace{1mm}\\
		\psi_1(\eta)\rightarrow -\infty\, \mbox{ as } \eta \rightarrow (-\lambda_{1})^-\,; \vspace{1mm}\\
		\psi_1'(\eta)=2\eta\sum\limits_{i=1}^{n}\dfrac{r_{i}^{2}}{(\lambda_{i}+\eta)^{3}}=2\eta\phi_1(\eta),\quad \mbox{where } \phi_1(\eta):=\sum\limits_{i=1}^{n}\dfrac{r_{i}^{2}}{(\lambda_{i}+\eta)^{3}};\vspace{1mm}\\ 
		\phi'_1(\eta)=-3\sum\limits_{i=1}^{n}\dfrac{r_{i}^{2}}{(\lambda_{i}+\eta)^{4}}<0,\,\forall \eta\in \left(-\lambda_2,-\lambda_1\right).
	\end{cases}
\end{equation}
Based on the negativeness of $\phi'_1$, one can obtain the following lemma, which tells us  that $\psi_1$ has at most three zero points on the interval $\left(-\lambda_2,-\lambda_1\right)$.

\begin{lemma}\label{th:psi-zerosnumber}
	Suppose that $\max\{-\lambda_2,0\}<-\lambda_1$ and $r_1\ne0$ for the problem \eqref{eq:s3M2}. When $-\lambda_2\geq0$, then either $\psi'_1(\eta)<0$   $\forall\eta\in\left(-\lambda_2,\, -\lambda_1\right)$, or $\exists\,\eta_0\in\left(-\lambda_2,\, -\lambda_1\right)$ such that
	\begin{equation} \label{eq:s3psi'sign-1}
		\psi'_1(\eta)
		\begin{cases}
			>0, & \forall\eta\in\left(-\lambda_2,\eta_0\right), \\
			=0, & \eta=\eta_0,\\
			<0,& \forall\eta\in\left(\eta_0,-\lambda_1\right).
		\end{cases}
	\end{equation}
	And when $-\lambda_2<0$, then the signs of $\psi'_1$ belong to one of the following four cases: 
	\begin{subequations} 
		\begin{align}
			\label{eq:s3psi'sign-2}
			\mbox{Case 1.}\quad&\psi'_1(\eta)
			\begin{cases}
				>0, & \forall\eta\in\left(-\lambda_2,0\right), \\
				=0, & \eta=0,\\
				<0,& \forall\eta\in\left(0,-\lambda_1\right);
			\end{cases}\\
			\vspace{1mm}
			\label{eq:s3psi'sign-3-0}
			\mbox{Case 2.}\quad&\psi'_1(\eta)
			\begin{cases}
				<0, & \forall\eta\in\left(-\lambda_2,0\right), \\
				=0, & \eta=0,\\
				<0,& \forall\eta\in\left(0,-\lambda_1\right);
			\end{cases}\\
			\vspace{1mm}
			\label{eq:s3psi'sign-3-1}
			\mbox{Case 3.}\quad&\psi'_1(\eta)
			\begin{cases}
				<0, & \forall\eta\in\left(-\lambda_2,\,\eta_0\right), \\
				=0, & \eta=\eta_0,\\
				>0, & \forall\eta\in\left(\eta_0,\, 0\right), \\
				=0, & \eta=0,\\
				<0,& \forall\eta\in\left(0,\, -\lambda_1 \right);
			\end{cases}\\
			\vspace{1mm}
			\label{eq:s3psi'sign-3-2}
			\mbox{Case 4.}\quad&\psi'_1(\eta)
			\begin{cases}
				<0, & \forall\eta\in\left(-\lambda_2,\,0\right), \\
				=0, & \eta=0,\\
				>0, & \forall\eta\in\left(0,\, \eta_0\right), \\
				=0, & \eta=\eta_0,\\
				<0,& \forall\eta\in\left(\eta_0,\, -\lambda_1 \right).
			\end{cases}
		\end{align}
	\end{subequations}
\end{lemma}

\begin{proof}
	As $r_1\ne0$ ensures that
	\begin{equation}\label{eq:s3phi'-sign}
		\phi'_1(\eta)=\sum_{i=1}^{n}\dfrac{-3r_{i}^{2}}{(\lambda_{i}+\eta)^{4}}\leq\dfrac{-3r_{1}^{2}}{(\lambda_{1}+\eta)^{4}}<0,\quad \forall \eta\in\left(-\lambda_2,\,-\lambda_1\right),
	\end{equation}
	the function $\phi_1$ is strictly monotone decreasing on the interval $\left(-\lambda_2,\,-\lambda_1\right)$. 
	Note that  $$\lim\limits_{\eta \rightarrow (-\lambda_{1})^{-}}\phi_1(\eta)= -\infty.$$
	So, either $\phi_1(\eta)<0$   $\forall\eta\in\left(-\lambda_2,\,-\lambda_1\right)$, or $\phi_1$ has a unique zero point $\eta_0\in\left(-\lambda_2,-\lambda_1\right)$ such that 
	\begin{equation} \label{eq:phi-eta0-sign}
		\phi_1(\eta)
		\begin{cases}
			>0, & \forall\eta\in\left(-\lambda_2,\eta_0\right); \\
			=0, & \eta=\eta_0;\\
			<0,& \forall\eta\in\left(\eta_0,-\lambda_1\right).
		\end{cases}
	\end{equation}
	When $-\lambda_2\geq0$, the desired conclusion is true straightforwardly because both  $\psi_1'(\eta)$ and  $\phi_1(\eta)$ have exactly the same signs on the interval $\left(-\lambda_2,\, -\lambda_1\right)$.  
	When $-\lambda_2<0$, \eqref{eq:s3psi'sign-2}, \eqref{eq:s3psi'sign-3-0}, \eqref{eq:s3psi'sign-3-1} and \eqref{eq:s3psi'sign-3-2} can directly follow  from the signs of  $\phi_1$ in accordance with the four cases: the non-existence of $\eta_{0}$, $\eta_{0}=0$, $\eta_{0}<0$ and $\eta_{0}>0$, respectively.  
\end{proof}

\begin{remark}\label{rmk:s3-psi1-pos-zeros}
	By Lemma \ref{th:psi-zerosnumber},  the signs of $\psi'_1$ on the interval $\left(\max\{-\lambda_2,0\},\,  -\lambda_1\right)$ satisfy that either $\psi'_1(\eta)<0$   $\forall\eta\in\left(\max\{-\lambda_2,0\},\, -\lambda_1\right)$, or $\exists\,\eta_0\in\left(\max\{-\lambda_2,0\},\, -\lambda_1\right)$ such that
	\begin{equation} \label{eq:s3psi'sign-3-3}
		\psi'_1(\eta)
		\begin{cases}
			>0, & \forall\eta\in\left(\max\{-\lambda_2,0\},\eta_0\right), \\
			=0, & \eta=\eta_0,\\
			<0,& \forall\eta\in\left(\eta_0,-\lambda_1\right).
		\end{cases}
	\end{equation}  
	So, on the interval $\left(\max\{-\lambda_2,0\},\,  -\lambda_1\right)$, the zero points of $\psi_1$  must belong to one of the following three cases: {\bf (i)} $\psi_1$ has no zero points; {\bf (ii)} $\psi_1$ has one zero point $\eta_1$ with $\psi_1'(\eta_1)\leq0$;  {\bf (iii)} $\psi_1$ has two zero points  $\eta_1$ and $\eta_2$ with $\psi_1'(\eta_1)<0$ and $\psi_1'(\eta_2)>0$, respectively. 
\end{remark}

In order to figure out whether or not the zero points of $\psi_1$ on the interval $\left(\max\{-\lambda_2,0\},  -\lambda_1\right)$  admit local-nonglobal minimizers of the problem \eqref{eq:s3M2}, one needs further to define two matrices $W_1(\eta)\in\mathcal{R}^{n\times(n-1)}$ and $B_1(\eta)\in\mathcal{R}^{(n-1)\times(n-1)}$ by
\begin{equation*}
	\begin{cases}
		W_1(\eta)=\dfrac{r_{1}}{\lambda_{1}+\eta}
		\begin{bmatrix}v_2,&v_3,&\cdots,&v_n \end{bmatrix}-v_1u^T, \vspace{2mm}\\
		B_1(\eta)=W_1(\eta)^{T}(A+\eta I)W_1(\eta)
		=\tilde{B}_1(\eta)+(\lambda_{1}+\eta)uu^T,
	\end{cases}
	\quad\forall\eta\in\left(-\lambda_2,-\lambda_1\right),
\end{equation*}
where 
\begin{equation*}
	\begin{array}{l}
		\tilde{B}_1(\eta)=\dfrac{r_1^2}{(\lambda_1+\eta)^2}\diag\left({\lambda_2+\eta},{\lambda_3+\eta},\cdots,{\lambda_n+\eta}\right) \mbox{ and }\\
		u=\left[\dfrac{r_{2}}{\lambda_{2}+\eta},\dfrac{r_{3}}{\lambda_{3}+\eta},\cdots,\dfrac{r_{n}}{\lambda_{n}+\eta}\right]^T.
	\end{array}
\end{equation*}
Similarly to the proof in Theorem 3.1 of \cite{martinez1994}, one can verify that,  under the assumptions that ``$\max\{-\lambda_2,0\}<-\lambda_1$" and ``$r_1\ne0$", for all $\eta\in \left(-\lambda_2,-\lambda_1\right)$ the matrices $W_1(\eta)$ and $B_1(\eta)$ satisfy the following properties:
\begin{equation}\label{eq:s3r(W)-detB}
	\begin{cases}
		\rank\left(W_1(\eta)\right)=n-1,\vspace{1mm}\\
		\nabla g\left(y_1(\eta)\right)^TW_1(\eta)=0\quad (\mbox{ by }\nabla g|_{y_1(\eta)}=-2\eta y_1(\eta)), \vspace{1mm}\\
		\det{\left( B_1(\eta)\right)} =h_1(\eta)\phi_1(\eta) \quad (\mbox{by Schur complement formula}),
	\end{cases}
\end{equation}
where $h_1(\eta)=\dfrac{r_{1}^{2n-4}(\lambda_{2}+\eta)\cdots(\lambda_{n}+\eta)}{(\lambda_{1}+\eta)^{2n-5}}<0$ and $\phi_1(\eta)$ is defined in \eqref{eq:s3nablag-psi'}. 
The following lemma shows that the positive definiteness of $B_1(\eta)$ can be determined by only the positiveness of its determinant.

\begin{lemma}\label{th:s3-B1-definite}
	Suppose that $\lambda_1<\lambda_2$ and $r_1\ne0$ for the problem \eqref{eq:s3M2}. Then for each $\eta\in\left(-\lambda_2, -\lambda_1\right)$, $B_1(\eta)\succ0$ if and only if $\phi_1(\eta)<0$.
\end{lemma}

\begin{proof}
	{\bf ``$\Longrightarrow$". } Apparently $B_1(\eta)\succ0$ $\Longrightarrow$ $\det{\left(B_1(\eta) \right)}>0$ $\Longrightarrow$ $\phi_1(\eta)<0$.
	
	{\noindent\bf ``$\Longleftarrow$". } Suppose by contradiction that there is a number ${\eta_1}\in\left(-\lambda_2, -\lambda_1\right)$ such that $\phi_1(\eta_1)<0$ but $B_1({\eta_1})\nsucc0$. Note that $\phi_1(\eta_1)<0$ means $\det{\left( B_1({\eta_1})\right)}>0$, which implies that $B_1({\eta_1})$ is nonsingular and must have at least one negative eigenvalue. On the other hand, one can easily verify that there is a sufficiently small positive number $\varepsilon>0$ such that $B_1({\eta})\succ0$ for all ${\eta}\in\left[-\lambda_1-\varepsilon, -\lambda_1\right)\subseteq\left(\eta_{1},\, -\lambda_1\right)$. Therefore, from the continuity of $B_1({\eta})$ on the closed interval $\left[\eta_{1},\, -\lambda_1-\varepsilon\right]$, there exists a number  $\eta_{2}\in\left(\eta_{1},\, -\lambda_1-\varepsilon\right)$ such that $$\det{\left( B_1({\eta_2})\right)}=0\,\Longrightarrow\, \phi_1(\eta_2)=0,$$ 
	which contradicts with $\phi_1(\eta_1)<0$ and $\eta_1<\eta_2$, because $\phi_1(\eta)$ is strictly monotone decreasing on the interval $(-\lambda_2,\,-\lambda_1)$ from \eqref{eq:s3phi'-sign}. 
\end{proof}

From Lemma \ref{th:s3-B1-definite},  one can obtain the following theorem, which indicates that any zero point $\eta^*$ of $\psi_1$ on the interval $\left(\max\{-\lambda_2,0\},  -\lambda_1\right)$ with $r_1\ne0$ admits a local-nonglobal minimizer of the problem \eqref{eq:s3M2} if and only if $\psi'_1(\eta^*)<0$. 

\begin{theorem}
	\label{th:s3-main-th1}
	The problem \eqref{eq:s3M2} has a local-nonglobal minimizer $y^*$ with $\mu^*={1}/{\eta^{*}}>0$ if and only if   $\max\{-\lambda_2,0\}<\eta^*<-\lambda_1$, $r_1\ne0$,  $y^*=y_1(\eta^*)$, $\psi_1(\eta^*)=0$ and $\psi_1'(\eta^*)<0$, where $y_1(\eta)$ and $\psi_1(\eta)$ are defined by \eqref{eq:y-psi(eta)}.
\end{theorem}

\begin{proof}
	{\bf ``$\Longrightarrow$". } Assume that $y^*$ is a local-nonglobal minimizer of  \eqref{eq:s3M2} with $\mu^*={1}/{\eta^{*}}>0$. By Lemma \ref{th:s3mu*range}, $y^*$ and $\eta^{*}$ satisfy \eqref{eq:y*eta*cond}, that is, 
	\begin{equation*}
		\max\{-\lambda_2,0\}<\eta^*<-\lambda_1,\,r_{1}\neq0,\, y^*=y_1(\eta^*)\, \mbox{ and }\, \psi_1(\eta^*)=g(y_1(\eta^*))=g(y^*)=0.
	\end{equation*}
	Moreover, \eqref{eq:s3r(W)-detB} guarantees that  $\mbox{Null}\left(\nabla g(y^*)^T\right)=\mbox{Range}\left(W_1(\eta^*)\right)$,  which yields from Theorems \ref{th:GTRE-strict-nonglobal-sufnec} and \ref{th:nonstrict-nonglobal-zero} that $B_1(\eta^{*})\succ0$. Then, by using Lemma \ref{th:s3-B1-definite}, one has 
	\begin{equation*}
		\phi_1(\eta^{*})<0\,\Longrightarrow\,\psi_1'(\eta^*)=2\eta^*\phi_1(\eta^*)<0.
	\end{equation*}
	
	{\noindent\bf ``$\Longleftarrow$". } Note that the condition ``$\max\{-\lambda_2,0\}<\eta^*<-\lambda_1$"  implies that $\lambda_1<0$, which  means that $A\ne0$. Then one can easily verify that the point $y^*=y_1(\eta^*)$ and the multiplier $\mu^*=1/\eta^*$ satisfy all the three conditions (i), (ii) and (iii) given in Theorem \ref{th:GTRE-strict-nonglobal-sufnec}, except for the sufficient optimality condition \eqref{eq:2SufCond}. From \eqref{eq:s3nablag-psi'} and \eqref{eq:s3r(W)-detB},  one obtains that $\nabla g(y^*)\ne0$ and $\mbox{Null}\left(\nabla g(y^*)^T\right)=\mbox{Range}\left(W_1(\eta^*)\right)$. So, if one verifies $B_1(\eta^*)=W_1(\eta^*)^{T}(A+\eta^* I)W_1(\eta^*)\succ0$, then \eqref{eq:2SufCond} holds. In fact, one deduces from $\max\{-\lambda_2,0\}<\eta^*$ and $\psi_1'(\eta^*)=2\eta^*\phi_1(\eta^*)<0$    that
	\begin{equation*}
		\phi_1(\eta^*)<0 \,\Longrightarrow\,B_1(\eta^{*})\succ0\,\,\, (\mbox{by Lemma \ref{th:s3-B1-definite}}).
	\end{equation*}
	Therefore, by Theorem \ref{th:GTRE-strict-nonglobal-sufnec}, $y^*=y_1(\eta^*)$ is a local-nonglobal minimizer of  \eqref{eq:s3M2} with $\mu^*={1}/{\eta^{*}}>0$.
\end{proof}

\begin{example} \label{ex:psi-roots}
	\begin{equation} \label{eq:psi-roots-3}
		\min\, \{y_{1}^{2}+y_{2}^{2}\,\,\,|\, g(y_{1},y_{2})=-4y_{1}^{2}+y_{2}^{2}+2y_{1}+16y_{2}+45=0\}. 
	\end{equation}
	This problem has $\lambda_{1}=-4$, $\lambda_{2}=1$, $r_{1}=1$ and $r_{2}=8$. As shown in Figure \ref{pic:psi(eta)-1}, the function 
	\begin{equation*}
		\psi_1(\eta)=\dfrac{-4}{(-4+\eta)^2}+ \dfrac{64}{(1+\eta)^2}-\dfrac{2}{-4+\eta}
		-\dfrac{128}{1+\eta}+45, \, \,\,\, \eta\in\left(-\lambda_2,-\lambda_1\right) =(-1,4),
	\end{equation*}
	has three zero points: $\eta_1=-0.3512$, $\eta_2=1.1829$ and $\eta_3=3.6018$ with $\psi'(\eta_1)<0$, $\psi'(\eta_2)>0$, $\psi'(\eta_3)<0$. According to Theorem \ref{th:s3-main-th1}, $\eta_3=3.6018$ admits a local-nonglobal minimizer of the problem \eqref{eq:psi-roots-3} with a positive multiplier $1/\eta_3$. 
\end{example}

\begin{figure}[htbp]
	\centering
	\begin{tikzpicture}[scale=0.65] 
		\begin{axis}[xlabel=$\eta$, ylabel=$\psi_1(\eta)$, xmin=-2, xmax=5, ymin=-40, ymax=40]
			
			\addplot[domain=-0.9:3.9, samples=100, blue, line width=1pt]{-4/(-4+x)^2+64/(1+x)^2-2*(1/(-4+x)+64/(1+x))+45};	
			
			\addplot[domain=-10:10, samples=100, red, dashed, line width=1pt]{0};
			
			\draw[dashed, line width=1pt] (-1, -50)--(-1, 50); 
			\draw[dashed, line width=1pt] (4, -50)--(4, 50); 
			\draw[dashed, line width=1pt] (0, -50)--(0, 50);
			
			\draw (-0.65,0)node[above=2mm,red]{$\eta_1$};
			\draw (-0.3512,0) [fill=red] circle (3pt);
			
			\draw (1.25,0)node[above=2mm,red]{$\eta_2$};
			\draw (1.1829,0) [fill=red] circle (3pt);
			
			\draw (3.79,0)node[above=2mm,red]{$\eta_3$};
			\draw (3.6018,0) [fill=red] circle (3pt);
		\end{axis}
	\end{tikzpicture}
	\caption{$\psi_1(\eta)$ has three zero points on $(-\lambda_{2},-\lambda_{1})$}
	\label{pic:psi(eta)-1}
\end{figure}
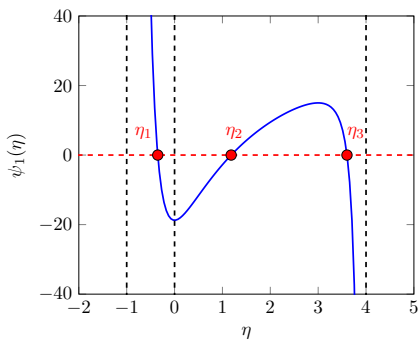

\begin{remark}
	If $\eta^*>0$  admits a local-nonglobal minimizer $y_1(\eta^*)$ of  \eqref{eq:s3M2} by Theorem \ref{th:s3-main-th1}, then $x^*=Q^Ty_1(\eta^*)+q=q-\sum\limits_{i=1}^{n}\dfrac{r_{i}}{\lambda_{i}+\eta^*}\,Q^Tv_i$ is just a local-nonglobal minimizer of $(GTRE)$  with the multiplier $\mu_1+1/\eta^*$. Moreover, if $\mu_1+1/\eta^*>0$, then $x^*$ is also a local-nonglobal minimizer of $(GTR)$.
\end{remark}

Now let us consider the case ``$\mu^*<0$". This case can be exactly regarded as the case ``$-\mu^*>0$" with the equality constraint ``$-g(y)=y^{T}(-A) y+2(-b)^{T}y+(-c)=0$". Thus one may define a vector function $y_2(\eta)$ and a function $\psi_2(\eta)$ by
\begin{equation}\label{eq:y2-psi2(eta)}
	\begin{cases}
		y_2(\eta):=-(-A+\eta I)^{-1}(-b)=-\sum\limits_{i=1}^{n}\dfrac{r_{i}}{\lambda_{i}-\eta}\,v_i,\vspace{1mm}\\
		\psi_2(\eta):=-g(y_2(\eta))=\sum\limits_{i=1}^{n} \dfrac{-\lambda_ir_{i}^{2}}{(\lambda_{i}-\eta)^{2}}+\sum\limits_{i=1}^{n} \dfrac{2r_{i}^{2}}{\lambda_{i}-\eta}-c,
	\end{cases}
	\forall\eta\in(\lambda_{n-1},\lambda_n).
\end{equation}
And then, applying Theorem \ref{th:s3-main-th1} to \eqref{eq:y2-psi2(eta)}, one gets the following result directly.

\begin{corollary}
	\label{th:s3-mu-negative-th1}
	$y^*$ is a local-nonglobal minimizer of the problem \eqref{eq:s3M2} with $\mu^{*}=-{1}/{\eta^*}<0$ if and only if   $\max\{ \lambda_{n-1},0\}<\eta^*< \lambda_n$, $r_n\ne0$,  $y^*=y_2(\eta^*)$, $\psi_2(\eta^*)=0$ and $\psi_2'(\eta^*)<0$.
\end{corollary}

Note that $y_1(\eta)$, $y_2(\eta)$, $\psi_1(\eta)$ and $\psi_2(\eta)$ have  relations as follows:
\begin{equation*}
	y_2(\eta)=y_1(-\eta)\,\,\mbox{and}\,\ \psi_2(\eta)=-\psi_1(-\eta),\quad\forall\eta\in\left(\lambda_{n-1},\lambda_n\right).
\end{equation*}
So Corollary \ref{th:s3-mu-negative-th1} can equivalently restated as follows, based on $y_1(\eta)$ and $\psi_1(\eta)$.

\begin{corollary}
	\label{th:s3-mu-negative-th1-1}
	$y^*$ is a local-nonglobal minimizer of the problem \eqref{eq:s3M2} with $\mu^{*}=-{1}/{\eta^*}<0$ if and only if  $\max\{ \lambda_{n-1},0\}<\eta^*<\lambda_n$, $r_n\ne0$,  $y^*=y_1(-\eta^*)$, $\psi_1(-\eta^*)=0$ and $\psi_1'(-\eta^*)<0$.
\end{corollary}

Consider Example \ref{ex:psi-roots} again. By using Corollary \ref{th:s3-mu-negative-th1-1}, one can verify that $\eta_1=-0.3512$ admits another local-nonglobal minimizer of the problem \eqref{eq:psi-roots-3}, which corresponds to the multiplier $1/\eta_1<0$. Therefore, the problem \eqref{eq:psi-roots-3} given in Example \ref{ex:psi-roots} has exactly two local-nonglobal minimizers, with the multiplier values $1/\eta_1<0$ and $1/\eta_3>0$, respectively.

Summarizing Lemmas \ref{th:s3mu*range} and \ref{th:psi-zerosnumber} (Remark \ref{rmk:s3-psi1-pos-zeros}), Theorem \ref{th:s3-main-th1} and Corollary \ref{th:s3-mu-negative-th1}, one obtains the following result, which involves the number of the local-nonglobal minimizers of $(GTRE)$ under the assumption that $\exists\mu_1\in\mathcal{R}$ such that $A_0+\mu_1A_1\succ0$.

\begin{theorem} \label{th:s3-number-th1}
	Suppose that $A_0$ and $A_1$ are jointly positive definite, that is,  $\exists\mu_1\in\mathcal{R}$ such that $A_0+\mu_1A_1\succ0$. Then the following  statements are true.\\
	{\bf (i)} If $A_1$ is semidefinite i.e. $A_1\succeq0$ or $A_1\preceq0$,  the problem $(GTRE)$ has at most one local-nonglobal minimizer.\\
	{\bf (ii)} If  $A_1$ is indefinite i.e. $\exists\bar{x},\tilde{x}\in\mathcal{R}^n$ such that $\bar{x}^TA_1\bar{x}<0 <\tilde{x}^TA_1\tilde{x}$, the problem $(GTRE)$ has at most two local-nonglobal minimizers.
\end{theorem}

\begin{proof}  
	According to the former analysis in this section, the assumption that ``$A_0+\mu_1A_1\succ0$" guarantees that $(GTRE)$ is equivalent to the problem \eqref{eq:s3M2}. Moreover, from Lemmas \ref{th:s3mu*range} and \ref{th:psi-zerosnumber} (Remark \ref{rmk:s3-psi1-pos-zeros}), Theorem \ref{th:s3-main-th1} and Corollary \ref{th:s3-mu-negative-th1}, the problem \eqref{eq:s3M2} has at most two local-nonglobal minimizers $y_1(\eta^*_1)$ and $y_2(\eta^*_2)$. One local-nonglobal minimizer $y_1(\eta^*_1)$ corresponds to $\lambda_{1}<0$, and the other $y_2(\eta^*_2)$ corresponds to $\lambda_n>0$, which ensures that both the statements (i) and (ii) are all true, because  $A$ and  $A_1$ have the same inertia indices duo to ${A}=QA_1Q^T$ and $\det(Q)\ne0$.
\end{proof}

\begin{remark}
	As we see, the problem \eqref{eq:psi-roots-3} in Example \ref{ex:psi-roots} has exactly two local-nonglobal minimizers with the multiplier values  $1/\eta_1\approx-2.8474$  and $1/\eta_3\approx0.2776$, respectively. So the upper bound of the number of the local-nonglobal minimizers given by Theorem \ref{th:s3-number-th1} is {\bf tight}. 
\end{remark}

Applying Theorems \ref{th:ineq-nonglobal=eq-nonglobal} and \ref{th:s3-number-th1} to $(GTR)$, one immediately gets the following corollary involving the number of the local-nonglobal minimizers of $(GTR)$.

\begin{corollary} \label{th:s3-number-th1-1}
	Suppose that  $A_0$ and $A_1$ are jointly positive definite. Then the following three statements are true.\\
	{\bf (i)} If $A_1$ is semidefinite i.e. $A_1\succeq0$ or $A_1\preceq0$,  the problem $(GTR)$ has at most one local-nonglobal minimizer.\\
	{\bf (ii)} If  $A_1$ is indefinite i.e. $\exists\bar{x},\tilde{x}\in\mathcal{R}^n$ such that $\bar{x}^TA_1\bar{x}<0 <\tilde{x}^TA_1\tilde{x}$, the problem $(GTR)$ has at most two local-nonglobal minimizers.
\end{corollary}

\begin{proof}
	If $n=1$, the desired conclusions are true apparently. And if $n\geq2$, the desired conclusions follow directly from Theorems \ref{th:ineq-nonglobal=eq-nonglobal} and \ref{th:s3-number-th1}.
\end{proof}

\begin{remark}
	The upper bound of the number of the local-nonglobal minimizers given by Corollary \ref{th:s3-number-th1-1} is also {\bf tight}. By slightly modifying the problem \eqref{eq:psi-roots-3}, one gets the following problem:
	\begin{equation} \label{pbm:psi-roots-3}
		\min\, \{x_{1}^{2}+x_{2}^{2}-4f_1(x_1,x_2)\,\,|\, f_1(x_{1},x_{2})=-4x_{1}^{2}+x_{2}^{2}+2x_{1}+16x_{2}+45\leq0\}. 
	\end{equation}
	Utilizing the problem \eqref{eq:psi-roots-3}, one can easily obtain that \eqref{pbm:psi-roots-3} has exactly two local-nonglobal minimizers with the multiplier values $4+1/\eta_1\approx1.1526$  and $4+1/\eta_3\approx4.2776$, respectively. 
\end{remark}

\subsection{On the problem \texorpdfstring{\eqref{eq:s3M3}}{}}

For the problem \eqref{eq:s3M3},  one obtains an analogy to Lemma \ref{th:s3mu*range} by following the proof of Lemma \ref{th:s3mu*range}.

\begin{lemma}\label{th:s3mu*range-minusy}
	Suppose that $n\geq2$. If $y^{*}$ is a local-nonglobal minimizer of the problem \eqref{eq:s3M3}, then $g(y^*)=0$ and there is a  unique Lagrangian multiplier $0\ne\mu^{*}\in\mathcal{R}$ such that
	\begin{equation*} %\label{eq:s3-first-nec-1}
		\left(-I+\mu^*{A}\right)y^*+\mu^{*}{b} =0.
	\end{equation*}  
	Moreover, if $\mu^{*}>0$ then   $\max\{\lambda_1,0\}<\dfrac{1}{|\mu^*|}<\lambda_2$ and  $r_{1}\neq0$; else if $\mu^{*}<0$ then  $\max\{-\lambda_n,0\}<\dfrac{1}{|\mu^*|}<-\lambda_{n-1}$ and $r_{n}\neq0$.
\end{lemma}

\begin{remark} \label{rmk:s3-ngeq2}
	In the above lemma, the assumption ``$\,n\geq2$" is necessary.  Consider the following problem $\min\{-y^2\,|\, y(y-1)=0\}$.
	It has a local-nonglobal minimizer $y^{*}=0$, but the corresponding Lagrangian multiplier $\mu^{*}=0$.
\end{remark}

Similarly to the problem \eqref{eq:s3M2}, our discussions focus first on  the local-nonglobal minimizers of the problem \eqref{eq:s3M3} with $\mu^*=1/\eta^*>0$.
Following \eqref{eq:y-psi(eta)}, one may define a vector function $y_3(\eta)$ and a function $\psi_3(\eta)$ by 
\begin{equation}\label{eq:y3-psi3(eta)}
	\begin{cases}
		y_3(\eta):=-(A-\eta I)^{-1}b=-\sum\limits_{i=1}^{n}\dfrac{r_{i}}{\lambda_{i}-\eta}\,v_i,\vspace{1mm}\\
		\psi_3(\eta):=g(y_3(\eta))=\sum\limits_{i=1}^{n} \dfrac{\lambda_ir_{i}^{2}}{(\lambda_{i}-\eta)^{2}}-\sum\limits_{i=1}^{n} \dfrac{2r_{i}^{2}}{\lambda_{i}-\eta}+c,
	\end{cases}
	\quad\forall\eta\in \left(\lambda_1,\lambda_2\right),
\end{equation}
where $\max\{\lambda_1,0\}<\lambda_2$ and $r_1\ne0$.  Both functions have the following properties:
\begin{equation} \label{eq:s3-psi3-property-1}
	\begin{cases}
		y_3(\eta)\ne0,\quad \forall \eta\in \left(\lambda_1,\lambda_2\right);\vspace{1mm}\\
		\nabla g|_{y_3(\eta)}=2\eta y_3(\eta)\ne0,\quad\forall \eta\in \left(\max\{\lambda_1,0\},\lambda_2\right);\vspace{1mm}\\
		\psi_3(\eta)\rightarrow
		\begin{cases}
			-\infty,\quad\mbox{ if } \lambda_1<0,\vspace{1mm}\\
			+\infty,\quad\mbox{ if } \lambda_1\geq0,
		\end{cases}
		\mbox{ as }{\eta \rightarrow \lambda_{1}^{+};} \vspace{1mm}\\
		\psi_3'(\eta)=2\eta\sum\limits_{i=1}^{n}\dfrac{r_{i}^{2}}{(\lambda_{i}-\eta)^{3}}=2\eta\phi_3(\eta), \quad \mbox{where } \phi_3(\eta):=\sum\limits_{i=1}^{n}\dfrac{r_{i}^{2}}{(\lambda_{i}-\eta)^{3}}; 
		\vspace{1mm}\\
		\phi_3'(\eta)=3\sum\limits_{i=1}^{n}\dfrac{r_{i}^{2}}{(\lambda_{i}-\eta)^{4}}>0,\,\forall\eta\in \left(\lambda_1,\lambda_2\right).
	\end{cases}
\end{equation}
Utilizing the positiveness of $\phi'_3(\eta)$, one obtains an analogy to Remark \ref{rmk:s3-psi1-pos-zeros}. 

\begin{lemma}\label{th:psi3-zerosnumber}
	Suppose that $n\geq2$, $\max\{\lambda_1,0\}<\lambda_2$ and $r_1\ne0$ for the problem \eqref{eq:s3M3}. Then, on the open interval  $\left(\max\{\lambda_1,0\},\, \lambda_2\right)$,  the signs of $\psi'_3$ must belong to one of the following three cases: {\bf (i)} $\psi'_3(\eta)<0$; {\bf (ii)} $\psi'_3(\eta)>0$; {\bf (iii)}   $\exists\eta_0\in\left(\max\{\lambda_1,0\},\lambda_2\right)$ such that 
	\begin{equation} \label{eq:s3psi3'sign-1}
		\psi'_3(\eta)
		\begin{cases}
			<0, & \forall\eta\in\left(\max\{\lambda_1,0\},\eta_0\right), \\
			=0, & \eta=\eta_0,\\
			>0,& \forall\eta\in\left(\eta_0, \lambda_2 \right).
		\end{cases}
	\end{equation}
\end{lemma}

\begin{proof} As $\psi_3'(\eta)=2\eta\phi_3(\eta)$, both $\psi_3'$ and $\phi_3$ have the same signs on the interval $\left(\max\{\lambda_1,0\},\, \lambda_2\right)$. 
	Note that $\phi_3'(\eta)>0$ for all $\eta\in\left(\lambda_1,\lambda_2\right)$. So,  on the open interval $\left(\lambda_1,\lambda_2\right)$,  $\phi_3$ is strictly monotone increasing with $\lim_{\eta\rightarrow\lambda_{1}^+}\phi_3(\eta)=-\infty$ and has at most one zero point. Firstly, if $\phi_3(\eta)$ has no zero points on the interval $\left(\lambda_1,\lambda_2\right)$ then $\phi_3(\eta)<0$ for all $\eta\in\left(\lambda_1,\lambda_2\right)$, which implies that $\psi'_3(\eta)<0$ for all    $\eta\in\left(\max\{\lambda_1,0\},\, \lambda_2\right)$. Secondly, if $\phi_3(\eta)$ has a zero point $\eta_0\in\left(\lambda_1,\lambda_2\right)$ with $\eta_0\leq0$, then $\phi_3(\eta)>0$ for all    $\eta\in\left(\eta_0,\, \lambda_2\right)$, which means that $\psi'_3(\eta)>0$ for all    $\eta\in\left(\max\{\lambda_1,0\},\, \lambda_2\right)=(0,\,\lambda_2)$. Finally, if $\phi_3(\eta)$ has one zero point $\eta_0\in\left(\lambda_1,\lambda_2\right)$ with $\eta_0>0$, then \eqref{eq:s3psi3'sign-1} holds apparently. 
\end{proof}

The above lemma shows that $\psi_3$ has at most two zero points on the interval $(\max\{\lambda_1,0\}, \lambda_2)$. Similar to the problem \eqref{eq:s3M2}, in order to figure out whether or not the zero points  admit local-nonglobal minimizers of the problem \eqref{eq:s3M3}, we need further to define $W_3(\eta)\in\mathcal{R}^{n\times(n-1)}$ and $B_3(\eta)\in\mathcal{R}^{(n-1)\times(n-1)}$ by
\begin{equation*}
	\begin{cases}
		W_3(\eta)=\dfrac{r_{1}}{\lambda_{1}-\eta}\begin{bmatrix}v_2,&v_3,&\cdots,&v_n \end{bmatrix}-v_1\tilde{u}^T, \vspace{2mm}\\
		B_3(\eta)=W_3(\eta)^{T}(A-\eta I)W_3(\eta)
		=\tilde{B}_3(\eta)+(\lambda_{1}-\eta)\tilde{u}\tilde{u}^T,
	\end{cases}
	\quad\forall  \eta\in \left(\lambda_1,\lambda_2\right),
\end{equation*}
where 
\begin{equation*}
	\begin{cases}
		\tilde{B}_3(\eta)=\dfrac{r_1^2}{(\lambda_1-\eta)^2}\diag\left({\lambda_2-\eta},{\lambda_3-\eta},\cdots,{\lambda_n-\eta}\right)\\ 
		\tilde{u}=\left[\dfrac{r_{2}}{\lambda_{2}-\eta},\dfrac{r_{3}}{\lambda_{3}-\eta},\cdots,\dfrac{r_{n}}{\lambda_{n}-\eta}\right]^T.
	\end{cases}
\end{equation*}
If only $\max\{\lambda_1,0\}<\lambda_2$ and $r_1\ne0$, the matrices $W_3(\eta)$ and $B_3(\eta)$ satisfy the following properties:
\begin{equation}\label{eq:s3r(W)-detB-1}
	\begin{cases}
		\rank\left(W_3(\eta)\right)=n-1,\vspace{1mm}\\
		\nabla g\left(y_3(\eta)\right)^TW_3(\eta)=0, \vspace{1mm}\\
		\det{\left( B_3(\eta)\right)} =h_3(\eta)\phi_3(\eta),
	\end{cases}
	\,\forall\eta\in \left(\lambda_1,\lambda_2\right),
\end{equation}
where $h_3(\eta)=\dfrac{r_{1}^{2n-4}(\lambda_{2}-\eta)\cdots(\lambda_{n}-\eta)}{(\lambda_{1}-\eta)^{2n-5}}<0$. From \eqref{eq:s3r(W)-detB-1}, one can yield the following lemma by repeating the proof process of Lemma \ref{th:s3-B1-definite}.

\begin{lemma}\label{th:s3-B3-definite}
	Suppose that $\lambda_1<\lambda_2$ and $r_1\ne0$ for the problem \eqref{eq:s3M3}. Then for each $\eta\in\left(\lambda_1, \lambda_2\right)$, $B_3(\eta)\succ0$ if and only if $\phi_3(\eta)<0$.
\end{lemma}

Based on Lemmas \ref{th:s3mu*range-minusy} and \ref{th:s3-B3-definite}, and by following the proof of Theorem \ref{th:s3-main-th1}, one can also obtain a  necessary and sufficient test condition for a local-nonglobal minimizer of the problem \eqref{eq:s3M3} with a positive multiplier. It is stated in the following theorem and the corresponding proof is omitted. 

\begin{theorem}
	\label{th:s3-main-th2} 
	Suppose that $n\geq2$.  The problem \eqref{eq:s3M3} has a local-nonglobal minimizer $y^*$ with $\mu^*={1}/{\eta^{*}}>0$  if and only if  $\max\{\lambda_1,0\}<\eta^*<\lambda_2$,  $r_1\ne0$, $y^*=y_3(\eta^*)$, $\psi_3(\eta^*)=0$ and $\psi_3'(\eta^*)<0$.
\end{theorem}

The local-nonglobal minimizers of the problem \eqref{eq:s3M3} are just the local-nonglobal maximizers of the problem \eqref{eq:s3M2}. Moreover, $y_1(\eta)$, $y_3(\eta)$, $\psi_1(\eta)$ and $\psi_3(\eta)$ have the following relations:
\begin{equation*}
	y_3(\eta)=y_1(-\eta)\,\,\mbox{and}\,\ \psi_3(\eta)=\psi_1(-\eta),\quad\forall\eta\in\left(\lambda_{1},\lambda_2\right).
\end{equation*}
So Theorem \ref{th:s3-main-th2} may be equivalently expressed by using $y_1(\eta)$ and $\psi_1(\eta)$. 

\begin{corollary}
	\label{th:s3-main-th2-1} 
	Suppose that $n\geq2$.  The problem \eqref{eq:s3M3} has a local-nonglobal minimizer $y^*$ with $\mu^*={1}/{\eta^{*}}>0$  if and only if $\max\{\lambda_1,0\}<\eta^*<\lambda_2$,  $r_1\ne0$, $y^*=y_1(-\eta^*)$, $\psi_1(-\eta^*)=0$ and $\psi_1'(-\eta^*)>0$.
\end{corollary}

Let us consider the local-nonglobal maximizers of Example \ref{ex:psi-roots}, that is, we consider the following problem.
\begin{example} \label{ex:psi-roots-1}
	\begin{equation} \label{eq:psi-roots-1-1}
		\min\, \{-y_{1}^{2}-y_{2}^{2}\,\,\,|\, g(y_{1},y_{2})=-4y_{1}^{2}+y_{2}^{2}+2y_{1}+16y_{2}+45=0\}. 
	\end{equation}
	Applying Corollary \ref{th:s3-main-th2-1} to the problem \eqref{eq:psi-roots-1-1}, one can obtain that it has no local-nonglobal minimizers with positive multipliers.
\end{example}

Similarly to the problem  \eqref{eq:s3M2}, the case ``$\mu^*<0$" in the problem  \eqref{eq:s3M3} can  be also regarded as the case ``$-\mu^*>0$" with the equality constraint ``$-g(y)=y^{T}(-A) y+2(-b)^{T}y+(-c)=0$". So one defines a vector function $y_4(\eta)$ and a function $\psi_4(\eta)$ by
\begin{equation}\label{eq:y4-psi4(eta)}
	\begin{cases}
		y_4(\eta):=-(-A-\eta I)^{-1}(-b)=y_1(\eta),\vspace{1mm}\\
		\psi_4(\eta):=-g(y_4(\eta))=-g(y_1(\eta))=-\psi_1(\eta),
	\end{cases}
	\quad\forall\eta\in\left(-\lambda_n,\,-\lambda_{n-1}\right).
\end{equation}
And then the following corollary is obtained by applying Theorem \ref{th:s3-main-th2} to \eqref{eq:y4-psi4(eta)}.

\begin{corollary}
	\label{th:s3-mu-negative-th2} 
	Suppose that $n\geq2$.  The problem \eqref{eq:s3M3} has a local-nonglobal minimizer $y^*$ with $\mu^*=-{1}/{\eta^{*}}<0$  if and only if  $\max\{-\lambda_n,0\}<\eta^*<-\lambda_{n-1}$,  $r_n\ne0$, $y^*=y_4(\eta^*)$, $\psi_4(\eta^*)=0$ and $\psi_4'(\eta^*)<0$.
\end{corollary}

Similarly to Corollary \ref{th:s3-main-th2-1}, the above corollary may be also restated by using $y_1(\eta)$ and $\psi_1(\eta)$. 

\begin{corollary}
	\label{th:s3-mu-negative-th2-1} 
	Suppose that $n\geq2$.  The problem \eqref{eq:s3M3} has a local-nonglobal minimizer $y^*$ with $\mu^*=-{1}/{\eta^{*}}<0$  if and only if  $\max\{-\lambda_n,0\}<\eta^*<-\lambda_{n-1}$,  $r_n\ne0$,  $y^*=y_1(\eta^*)$, $\psi_1(\eta^*)=0$ and $\psi_1'(\eta^*)>0$.
\end{corollary}

Applying Corollary \ref{th:s3-mu-negative-th2-1} to Example \ref{ex:psi-roots-1}, one can verify that $\eta_2=1.1829$ admits a local-nonglobal minimizer of Example \ref{ex:psi-roots-1} with the multiplier $-1/\eta_2<0$.

Combining Lemmas \ref{th:s3mu*range-minusy} and \ref{th:psi3-zerosnumber}, Theorem \ref{th:s3-main-th2}, Corollaries \ref{th:s3-main-th2-1}, \ref{th:s3-mu-negative-th2} and \ref{th:s3-mu-negative-th2-1}, one can obtain a  novel result as follows: under the assumption that $\exists\mu_1\in\mathcal{R}$ such that $A_0+\mu_1A_1\prec0$, the problem $(GTRE)$ has at most one local-nonglobal minimizer.

\begin{theorem} \label{th:s3-number-th2} 
	Suppose that $A_0$ and $A_1$ are jointly negative definite, that is,  $\exists\mu_1\in\mathcal{R}$ such that $A_0+\mu_1A_1\prec0$. Then the problem $(GTRE)$ has at most one local-nonglobal minimizer.
\end{theorem}

\begin{proof}
	If $n=1$, the desired conclusion is true apparently.  So we need only to consider the case: ``$n\geq2$". Note that, under the assumption that $\exists\mu_1\in\mathcal{R}$ such that $A_0+\mu_1A_1\prec0$, the problem $(GTRE)$  can be exactly transformed into the problem \eqref{eq:s3M3}. Thus one needs only to prove that,  at the case of ``$n\geq2$", the problem \eqref{eq:s3M3} has at most one local-nonglobal minimizer. 
	
	From Lemmas \ref{th:s3mu*range-minusy} and \ref{th:psi3-zerosnumber}, Theorem \ref{th:s3-main-th2} and Corollary \ref{th:s3-mu-negative-th2}, the problem \eqref{eq:s3M3} has possibly at most two  local-nonglobal minimizers $y_3(\eta^*_3)$ and $y_4(\eta^*_4)$.  Wherein they have to satisfy that 
	\begin{equation}
		\label{eq:s3-p2-lambda2-pos-neg}
		\begin{array}{l}
			r_1\ne0,\, \max\{\lambda_1,0\}<\eta^*_3<\lambda_2\leq \cdots\leq\lambda_n,\,
			\psi_3(\eta^*_3)=0,\, \psi'_3(\eta^*_3)<0, \mbox{ and} \vspace{1mm}\\
			r_n\ne0,\, \max\{-\lambda_n,0\}<\eta^*_4<-\lambda_{n-1} \leq\cdots\leq-\lambda_1, \psi_4(\eta^*_4)=0, \psi'_4(\eta^*_4)<0,
		\end{array}
	\end{equation}
	which implies $\lambda_2>0$ and $\lambda_{n-1}<0$.
	If $n\geq3$, \eqref{eq:s3-p2-lambda2-pos-neg}  makes $\lambda_{2}$ satisfy $0<\lambda_{2}<0$,
	which is a contradiction.
	If $n=2$, by using Corollaries \ref{th:s3-main-th2-1} and \ref{th:s3-mu-negative-th2-1},  \eqref{eq:s3-p2-lambda2-pos-neg} can be reformulated into $r_1r_2\ne0$, $-\lambda_{2}<-\eta^*_3<0 <\eta^*_4<-\lambda_{1}$, $\psi_1(-\eta^*_3)=\psi_1(\eta^*_4)=0$, $\psi'_1(-\eta^*_3)>0$ and $\psi'_1(\eta^*_4)>0$,
	which is also a contradiction according to Lemma \ref{th:psi-zerosnumber}. Thus it is impossible that there exist both $y_3(\eta^*_3)$ and $y_4(\eta^*_4)$  simultaneously.
\end{proof}

Combining Theorems \ref{th:ineq-nonglobal=eq-nonglobal} and \ref{th:s3-number-th2}, one immediately gets the following corollary.

\begin{corollary} \label{th:s3-number-th2-1}
	Suppose that  $A_0$ and $A_1$ are jointly negative definite. Then the problem $(GTR)$ has at most one local-nonglobal minimizer.
\end{corollary}

\begin{proof}
	If $n=1$, the desired conclusion is true apparently. And if $n\geq2$, the desired conclusion follows directly from Theorems \ref{th:ineq-nonglobal=eq-nonglobal} and \ref{th:s3-number-th2}.
\end{proof}

\begin{remark}
	If $A_1$ is an identity matrix, then the results in Theorem \ref{th:s3-number-th2} and Corollary \ref{th:s3-number-th2-1} are reduced to be just the Mart\'{i}nez's \cite{martinez1994}. So the upper bound of the number of the local-nonglobal minimizers given by Theorem \ref{th:s3-number-th2} and Corollary \ref{th:s3-number-th2-1} is {\bf tight}. 
\end{remark}

\section{Computation of the local-nonglobal minimizers} \label{sec-computation}
In this section, we discuss how to compute all the local-nonglobal minimizers of  $(GTRE)$ with $n\geq2$ under the joint definiteness condition  (the local-nonglobal minimizers of $(GTR)$ are also found by Theorem \ref{th:ineq-nonglobal=eq-nonglobal}). Firstly, for any symmetric matrices $A_0$ and $A_1$, one can use Guo-Higham-Tisseur's algorithm \cite{guo2010} to determine whether they both are jointly definite or not.  If their joint definiteness holds, one finds $A_0+\mu_{1}A_1={sgn}\,LL^T$ from Guo-Higham-Tisseur's algorithm, where $L$ is  a lower triangular matrix with positive diagonal elements and ${sgn}$ is a sign number: ${sgn}=1$ or $-1$. Secondly, by using the formula \eqref{eq:s3-Qq-trans}, one transforms equivalently  $(GTRE)$ into the problem \eqref{eq:s3M2} (or \eqref{eq:s3M3}). Thirdly, one solves \eqref{eq:s3M2} (or \eqref{eq:s3M3}) to find all its local-nonglobal minimizers $y^*$. Finally, one retransforms $y^*$ into $x^*$ by using the formula \eqref{eq:s3-Qq-trans}. So the key issue is how to compute all the local-nonglobal minimizers of \eqref{eq:s3M2} (or \eqref{eq:s3M3}).

In order to compute the local-nonglobal minimizers of the problems \eqref{eq:s3M2} and \eqref{eq:s3M3}, we present two lemmas to provide a theoretical basis for our algorithm design. 

\begin{lemma}
	\label{th:s4-psieta1>0}
	Suppose that $\max\{-\lambda_2,0\}<-\lambda_1$ and $r_1\ne0$. Then the problem \eqref{eq:s3M2} has a local-nonglobal minimizer  with $\mu^{*}>0$ $\Longleftrightarrow$  $\exists\,\tilde{\eta}\in\left(\max\{-\lambda_2,0\},\,-\lambda_1\right)$ such that $\psi_1(\tilde{\eta})>0$, where $\psi_1(\eta)$ is defined in \eqref{eq:y-psi(eta)}.
\end{lemma}

\begin{proof}
	{\bf ``$\Longrightarrow$". } Assume that the problem \eqref{eq:s3M2} has  a local-nonglobal minimizer $y^*$  with $\mu^*={1}/{\eta^{*}}>0$.  From Theorem \ref{th:s3-main-th1}, it holds that ${\eta^*}\in\left(\max\{-\lambda_2,0\},\,-\lambda_1\right)$, $y^*=y_1(\eta^*)$, $\psi_1(\eta^*)=0$ and $\psi'_1(\eta^*)<0$. Then by Lemma \ref{th:psi-zerosnumber} (Remark \ref{rmk:s3-psi1-pos-zeros}), $\exists\epsilon>0$ such that $\psi'_1(\eta)<0$ for all $\eta\in[\eta^*-\epsilon,\,-\lambda_{1})\subseteq\left(\max\{-\lambda_2,0\},\,-\lambda_1\right)$, which implies that $\psi_1(\eta^*-\epsilon)>0=\psi_1(\eta^*)$. So $\tilde{\eta}:=\eta^*-\epsilon$ is just the desired.
	
	{\noindent\bf ``$\Longleftarrow$". } Assume that $\exists\,\tilde{\eta}\in\left(\max\{-\lambda_2,0\},\,-\lambda_1\right)$ such that $\psi_1(\tilde{\eta})>0$. Note from \eqref{eq:s3nablag-psi'} that $\psi_1(\eta)\longrightarrow-\infty$ as $\eta\longrightarrow (-\lambda_1)^-$. Thus, from the continuity of $\psi_1(\eta)$ on $\left[\tilde{\eta},\,-\lambda_1\right)$, there exists a $\eta^*\in\left(\tilde{\eta},\,-\lambda_1\right)$ such that $\psi_1(\eta^*)=0$. If one can verify that $\psi'_1(\eta^*)<0$, then by Theorem \ref{th:s3-main-th1} the point $y_1(\eta^*)$ is just a local-nonglobal minimizer of the problem \eqref{eq:s3M2} with $\mu^*={1}/{\eta^{*}}>0$, and the proof is completed. In fact, we suppose by contradiction that $\psi'_1(\eta^*)\geq0$. It, together with $\eta^*>0$, implies from Lemma \ref{th:psi-zerosnumber} (Remark \ref{rmk:s3-psi1-pos-zeros}) that \eqref{eq:s3psi'sign-3-3} holds and $\max\{-\lambda_2,0\}<\tilde{\eta}<\eta^*\leq\eta_0<-\lambda_1$,  which leads to $\psi'_1(\eta)>0$ $\forall\eta\in(\max\{-\lambda_2,0\},\,\eta_0)$  by using \eqref{eq:s3psi'sign-3-3}. So one has $0<\psi_1(\tilde{\eta})<0=\psi_1(\eta^*)$, which is a contradiction.
\end{proof}

Following the statement and proof of Lemma \ref{th:s4-psieta1>0}, one can also find a similar result about the problem \eqref{eq:s3M3}.

\begin{lemma}
	\label{th:s4-psi3-eta3-positive}
	Suppose that  $n\geq2$, $\max\{\lambda_1,0\}<\lambda_2$ and $r_1\ne0$.  Then the problem \eqref{eq:s3M3} has a local-nonglobal minimizer  with $\mu^{*}>0$ $\Longleftrightarrow$  $\exists\,\tilde{\eta}\in\left(\max\{\lambda_1,0\},\,\lambda_2\right)$ such that $\psi_3(\tilde{\eta})<0$, and if $\lambda_1<0$ there is $\psi_3(0)>0$, where $\psi_3(\eta)$ is defined in \eqref{eq:y3-psi3(eta)}. 
\end{lemma}

Now we describe our algorithm. 

\begin{breakablealgorithm} 
	\caption{The main routine} \label{alg:main}
	\begin{algorithmic}[1]
		\Require $A\in\mathcal{S}^n$, $b\in\mathcal{R}^{n}$, $c\in\mathcal{R}$, a sign number ${sgn}$. 
		
		\Ensure ${\mu}_k^*\in\mathcal{R}$ and ${y}^*_{k}\in\mathcal{R}^n$, $k=1,2$.  
		
		\State Make an eigenvalue decomposition of $A$:  $A=V\Lambda V^{T}$, $\Lambda=\mbox{diag}(\lambda_{1},\lambda_{2},$ $\cdots,\lambda_{n})$ with $\lambda_1\leq\lambda_2\leq\cdots,\leq\lambda_{n}$ and $V=[v_1,v_2,\cdots,v_n]$ with $V^TV=I$. Put $r_{i}=v_{i}^{T}b$ for $i=1,2,\cdots,n$. If ${sgn}=-1$, go to Line 4.
		
		\State  Set $U_1=\max\{-\lambda_2,0\}$, $U_2=-\lambda_{1}$, $r=r_{1}$, $\psi(\eta)=\psi_1(\eta)$ and $\psi'(\eta)=\psi'_1(\eta)$.
		Call Subalgorithm \ref{alg:s4-2} and return $\eta^{*}$. If $\eta^{*}={\rm NaN}$, put ${\mu}^*_1={\rm NaN}$ and ${y}^*_1={\rm NaN}$. Otherwise,  put ${\mu}^*_1=1/\eta^*$ and ${y}^*_1=-\sum\limits_{i=1}^{n}\dfrac{r_{i}}{\lambda_{i}+\eta^*}v_i$.  
		
		\State Set $U_1=\max\{\lambda_{n-1},0\}$ and $U_2=\lambda_{n}$, $r=r_{n}$, $\psi(\eta)=\psi_2(\eta)$ and $\psi'(\eta)=\psi'_2(\eta)$.  Call Subalgorithm \ref{alg:s4-2} and return $\eta^{*}$. If $\eta^{*}={\rm NaN}$, put ${\mu}^*_2={\rm NaN}$ and ${y}^*_{2}={\rm NaN}$, go to Line 6. Otherwise, put ${\mu}^*_2=-1/\eta^*$ and ${y}^*_{2}=-\sum\limits_{i=1}^{n} \dfrac{r_{i}}{\lambda_{i}-\eta^*} v_i$, go to Line 6. 
		
		\State Set $U_1=\max\{\lambda_1,0\}$, $U_2=\lambda_{2}$, $r=r_{1}$, $\lambda=\lambda_1$, $\psi(\eta)=\psi_3(\eta)$ and $\psi'(\eta)=\psi'_3(\eta)$. Call Subalgorithm \ref{alg:s4-3} and return $\eta^{*}$. If $\eta^{*}={\rm NaN}$, put ${\mu}^*_1={\rm NaN}$ and ${y}^*_{1}={\rm NaN}$. Otherwise, put ${\mu}^*_1=1/\eta^*$, ${y}^*_{1}=-\sum\limits_{i=1}^{n} \dfrac{r_{i}}{\lambda_{i}-\eta^*}v_i$, ${\mu}^*_2={\rm NaN}$, ${y}^*_2={\rm NaN}$, and then go to Line 6.
		
		\State Set $U_1=\max\{-\lambda_n,0\}$, $U_2=-\lambda_{n-1}$, $r=r_{n}$, $\lambda=\lambda_n$, $\psi(\eta)=\psi_4(\eta)$ and $\psi'(\eta)=\psi'_4(\eta)$. Call Subalgorithm \ref{alg:s4-3} and return $\eta^{*}$. If $\eta^{*}={\rm NaN}$, put ${\mu}^*_2={\rm NaN}$ and ${y}^*_2={\rm NaN}$.  Otherwise, put ${\mu}^*_2=-1/\eta^*$, ${y}^*_2=-\sum\limits_{i=1}^{n} \dfrac{r_{i}}{\lambda_{i}+\eta^*}v_i$.   
		
		\State Output ${\mu}^*_1$ and ${y}^*_1$, ${\mu}^*_2$ and ${y}^*_2$, and then stop.
	\end{algorithmic}
\end{breakablealgorithm}

In Subalgorithm \ref{alg:s4-2}, bisection method is used on $( \max \{-\lambda_2,0\}, -\lambda_{1})$ and $\left(\max\{\lambda_{n-1},0\},\,\lambda_{n}\right)$. 

\begin{breakablealgorithm1}
	\caption{For the problem \eqref{eq:s3M2}}
	
	\label{alg:s4-2}
	
	\begin{algorithmic}[1]
		
		\Require $\psi(\eta)$, $\psi'(\eta)$, $U_1$, $U_2$, $r$ and a tolerance $\epsilon>0$.
		
		\Ensure  ${\eta}^*$.  
		
		\State Ask $U_1<U_2$ and $r\neq0$? No, put $\eta^*={\rm NaN}$, output $\eta^{*}$ and then stop. Yes, set $temp=0$.

		\State Set $\eta=\frac{U_1+U_2}{2}$, and compute $\psi(\eta)$. 
		
		\State If $\psi(\eta)>0$, update $temp=1$ and $U_1=\eta$, go to Line 6; else if $\psi(\eta)=0$, go to Line 4; else i.e. $\psi(\eta)<0$, go to Line 5.
		
		\State Compute $\psi'(\eta)$. If $\psi'(\eta)>0$, update $temp=1$ and $U_1=\eta$, go to Line 6; else if $\psi'(\eta)=0$, put $\eta^{*}={\rm NaN}$, output $\eta^{*}$ and then stop; else i.e. $\psi'(\eta)<0$, put $\eta^{*}=\eta$, output $\eta^{*}$ and then stop.
		
		\State Compute $\psi'(\eta)$. If $\psi'(\eta)>0$, update $U_1=\eta$, go to Line 6; else if $\psi'(\eta)=0$,  put $\eta^{*}={\rm NaN}$, output $\eta^{*}$ and then stop; else i.e. $\psi'(\eta)<0$, update $U_2=\eta$, go to Line 6.
		
		\State If $U_2-U_1> \epsilon$, then go to Line 2. Otherwise, if $temp=0$,  put $\eta^{*}={\rm NaN}$, output $\eta^{*}$ and then stop; else i.e. $temp=1$, put $\eta^{*}=U_2$, output $\eta^{*}$ and then stop.
	\end{algorithmic}
\end{breakablealgorithm1}

In Subalgorithm \ref{alg:s4-3}, bisection method is used on $\left( \max \{\lambda_1,0\}, \,\lambda_{2}\right)$ and $\left(\max\{-\lambda_{n},0\},\,-\lambda_{n-1}\right)$. 

\begin{breakablealgorithm1} 
	\caption{For the problem \eqref{eq:s3M3}} 
	
	\label{alg:s4-3}
	
	\begin{algorithmic}[1]
		\Require $\psi(\eta)$, $\psi'(\eta)$, $U_1$, $U_2$, $r$, $\lambda$ and a tolerance $\epsilon>0$.
		
		\Ensure $\eta^{*}$. 
		
		\State  Ask $U_1<U_2$ and $r\neq0$? No, put $\eta^*={\rm NaN}$, output $\eta^{*}$ and then stop. Yes, if $\lambda<0$ then ask $\psi(0)>0$? no,  put $\eta^*={\rm NaN}$, output $\eta^{*}$ and then stop; yes, set $temp=0$.
		
		\State Set $\eta=\frac{U_1+U_2}{2}$, and compute $\psi(\eta)$. 
		
		\State If $\psi(\eta)<0$, update $temp=1$ and $U_2=\eta$, go to Line 6; else if $\psi(\eta)=0$, then go to Line 4; else i.e. $\psi(\eta)>0$, go to Line 5.
		
		\State Compute $\psi'(\eta)$. If $\psi'(\eta)>0$, update $temp=1$ and $U_2=\eta$, go to Line 6; else if $\psi'(\eta)=0$, put $\eta^*={\rm NaN}$, output $\eta^{*}$ and then stop; else i.e. $\psi'(\eta)<0$, put $\eta^{*}=\eta$, output $\eta^{*}$ and then stop.
		
		\State Compute $\psi'(\eta)$. If $\psi'(\eta)>0$, update $U_2=\eta$, go to Line 6; else if $\psi'(\eta)=0$, put $\eta^*={\rm NaN}$, output $\eta^{*}$ and then stop; else i.e. $\psi'(\eta)<0$, update $U_1=\eta$, go to Line 6.
		
		\State If $U_2-U_1>\epsilon$, go to Line 2. Otherwise, if $temp=0$, put $\eta^*={\rm NaN}$, output $\eta^{*}$ and then stop; else i.e. $temp=1$, put $\eta^{*}=U_1$, output $\eta^{*}$ and then stop.
	\end{algorithmic}
\end{breakablealgorithm1}

Here we present preliminary numerical experiments to show the performance of our algorithms {\bf (due to limited space, we can only show some simple numerical results of the problem \eqref{eq:s3M2}, but the problem \eqref{eq:s3M3} has really similar numerical results)}. All the elements of the matrix $A$ and the vector $b$ are randomly generated on the interval $[-100,100]$. However, the constant $c$ is adaptively adjusted to ensure that $\psi_1(\eta)$ has zeros on the interval $(-\lambda_2, -\lambda_1)$.  The tolerance $\epsilon$ is always taken as $\epsilon=1e-5$.  

All the experiments were performed in Python 3.12 on a Laptop with an Intel Core i5-11320H (3.20GHz) processor and 16 GB of RAM. For each $n$,  we generate 1000 examples, and record the average values of some indices. Wherein ``Eig{\_}time(s)", ``Bis{\_}time(s)", ``Bis{\_}iter" and ``Num{\_}LNGM" denote the running time of the eigenvalue decomposition of $A$, the running time and the iteration number of  Subalgorithm \ref{alg:s4-2},  and the number of  local-nonglobal minimizers of the problem \eqref{eq:s3M2}.
Certainly, the average running time of Algorithm \ref{alg:main} is roughly equal to the sum of Eig{\_}time and Bis{\_}time. 

\begin{table}[htbp]
	\setlength{\tabcolsep}{15pt}
	\centering 
	\caption{Numerical results of the problem \eqref{eq:s3M2} with $n$ from $100$ to $10000$} \label{table:n=100:10000} 
	\begin{tabular}{ccccc}
		\toprule 		
		$n$ & Eig{\_}time(s) & Bis{\_}time(s) & Bis{\_}iter & Num{\_}LNGM  \\
		
		\midrule
		100 & 4.24e-3 & 4.66e-4 &45.03& 1.000\\
		
		200 & 2.05e-2 & 6.21e-4 & 44.91&1.000\\
		
		300 & 4.56e-2 & 8.81e-4 & 44.79 &1.000\\
		
		400 & 8.18e-2 & 7.45e-4 & 44.38 &1.000\\
		
		500 & 0.13 & 8.81e-4 & 44.48 & 1.000\\
		
		600 & 0.21 & 1.27e-3 & 44.15 & 1.000\\
		
		700 & 0.27 & 1.55e-3 & 44.31 & 1.000 \\
		
		800 & 0.42 & 1.42e-3 & 44.01 & 1.000\\
		
		900 & 0.51 & 1.47e-3 & 44.03 & 1.000\\
		
		1000 & 0.58 & 1.35e-3 & 44.16 & 1.000\\
		
		2000 & 4.19 & 3.57e-3 & 43.85 & 1.000\\
		
		3000 & 9.84 & 3.25e-3 & 43.34 & 1.000\\
		
		4000 & 20.76 & 4.06e-3 & 43.37 & 1.000\\
		
		5000 & 35.54 & 4.65e-3 & 43.23 & 1.000\\
		
		6000 & 65.29 & 5.51e-3 & 43.14 & 1.000\\
		
		7000 & 99.41 & 6.33e-3 & 43.30 & 1.000\\
		
		8000 & 143.12 & 7.18e-3 & 43.16 &1.000\\
		
		9000 & 345.09 & 1.04e-2 & 43.06& 1.000\\
		
		10000 & 422.81 & 1.07e-2 & 42.83 & 1.000\\ 
		 \bottomrule
	\end{tabular}
\end{table}

In Table \ref{table:n=100:10000}, one can easily observe that, the running time of Algorithm \ref{alg:main}  comes mainly from Eig{\_}time, and  Subalgorithm \ref{alg:s4-2} is insensitive to the dimension $n$. In fact, the computational cost of Subalgorithm \ref{alg:s4-2} comes mainly from evaluating $\psi(\eta)$ and $\psi'(\eta)$, which is equal to $O(n)$. 

\section*{Declarations}
{\bf Conflict of interest } The authors declare that they have no conflict of interest.


\begin{thebibliography}{1}
	\bibitem{adachi2019} {S. Adachi and Y. Nakatsukasa}, {\em Eigenvalue-based algorithm and analysis for nonconvex QCQP with one constraint}, {Math. Program.}, 173 (2019), pp. 79--116.
	
	\bibitem{bazaraa2006} {M.S. Bazaraa, H. D. Sherali and C. M. Shetty}, {\em Nonlinear Programming: Theory and Algorithms 3rd}, 2006.
	
	\bibitem{beck2017} {A. Beck and D. Pan}, {\em A branch and bound algorithm for nonconvex quadratic optimization with ball and linear constraints}, {J Glob. Optim.}, 69 (2017), pp. 309--342.
	
	\bibitem{bienstock2016} {D. Bienstock}, {\em A note on polynomial solvability of the CDT problem}, {SIAM J. Optim.}, 26(1) (2016), pp. 488--498.
	
	\bibitem{bienstock2014} {D. Bienstock and A. Michalka}, {\em Polynomial solvability of variants of the trust-region subproblem}, {Proceedings of the twenty-fifth annual ACM-SIAM symposium on Discrete algorithms}, (2014), pp. 380--390.
		
	\bibitem{bomze2015} {I. M. Bomze and M. L. Overton}, {\em Narrowing the difficulty gap for the Celis-Dennis-Tapia problem}, {Math. Program.}, 151 (2015), pp. 459--476.
			
	\bibitem{consolini2023} {L. Consolini and M. Locatelli}, {\em Sharp and fast bounds for the Celis-Dennis-Tapia problem}, {SIAM J. Optim.}, 33(2) (2023), pp. 868--898.
	
	\bibitem{deng2020} {Z. Deng, C. Lu, Y. Tian and J. Luo}, {\em Globally solving extended trust region subproblems with two intersecting cuts}, {Optim. Lett.}, 14 (2020), pp. 1855--1867.
	
	\bibitem{fortin2005} {C. Fortin}, {\em Computing the local-nonglobal minimizer of a large scale trust-region subproblem}, {SIAM J. Optim.}, 16(1) (2005), pp. 263--296.
	
	\bibitem{guo2010} {C. Guo, N. J. Higham, and F. Tisseur}, {\em An improved arc algorithm for detecting definite hermitian pairs}, {SIAM J. Matrix Anal. Appl.}, 31(3) (2010), pp. 1131--1151.
	
	\bibitem{xia2013} {Y. Hsia and R. L.  Sheu}, {\em Trust region subproblem with a fixed number of additional linear inequality constraints has polynomial complexity}, (2013), {arXiv:1312.1398}.
	
	\bibitem{jiang2022} {R. Jiang and D. Li}, {\em Exactness conditions for semidefinite programming relaxations of generalization of the extended trust region subproblem}, {Mathematics of Operations Research}, 48(3) (2022), pp. 1235--1253.
	
	\bibitem{lucidi1998} {S. Lucidi, L. Palagi, and M. Roma}, {\em On some properties of quadratic programs with a convex quadratic constraint}, {SIAM J. Optim.}, 8(1) (1998), pp. 105--122.
	
	\bibitem{luo2024} {H. Luo, Y. Chen, X. Zhang, D. Li, and H. Wu,}, {\em Effective algorithms for optimal portfolio deleveraging problem with cross impact}, {Mathematical Finance}, 34(1) (2024), pp. 36--89.
	
	\bibitem{martinez1994} {J. M. Mart\'{i}nez}, {\em Local minimizers of quadratic functions on euclidean balls and spheres}, {SIAM J. Optim}, 4(1) (1994), pp. 159--176.
	
	\bibitem{more1993} {J. J. Mor\'{e}}, {\em Generalizations of the trust region problem}, {Optim. Methods Softw.}, 2 (1993), pp. 189--209.
	
	\bibitem{polik2007} {I. P\'{o}lik and T. Terlaky}, {\em A survey of the S-lemma}, {SIAM Rev.}, 49(3) (2007), pp. 317--418.
	
	\bibitem{rontsis2022} {N. Rontsis, P. J. Goulart, and Y. Nakatsukasa}, {\em An active-set algorithm for norm constrained quadratic problems}, {Math. Program.}, 193 (2022), pp. 447--483.
	
	\bibitem{salahi2017} {M. Salahi, A. Taati, and  H. Wolkowicz}, {\em Local nonglobal minima for solving large-scale extended trust-region subproblems}, {Comput. Optim. Appl.}, 66 (2017), pp. 223--244.
		
	\bibitem{xia2023} {M. Song, J. Wang, H. Liu, and Y. Xia}, {\em On local minimizers of nonconvex homogeneous quadratically constrained quadratic optimization with at most two constraints}, {SIAM J. Optim.}, 33(1) (2023), pp. 267--293.
	
	\bibitem{taati2020} {A. Taati and M. Salahi}, {\em On local-nonglobal minimizers of quadratic optimization problem with a single quadratic constraint}, {Numer. Funct. Anal. Optim.}, 41(8) (2020), pp. 969--1005.
	
	\bibitem{xia2021} {J. Wang, M. Song, and Y. Xia}, {\em On local minimizers of generalized trust-region subproblem}, (2021), {arXiv:2108.13729}.
	
	\bibitem{xia2022} {J. Wang, M. Song, and Y. Xia}, {\em On local nonglobal minimum of trust-region subproblem and extension}, {J. Optim. Theory Appl.}, 195 (2022), pp. 702--722.
	
	\bibitem{xia2020} {J. Wang and Y. Xia}, {\em Closing the gap between necessary and sufficient conditions for local nonglobal minimizer of trust region subproblem}, {SIAM J. Optim.}, 30(3) (2020), pp. 1980--1995.
	
	\bibitem{xia2016} {Y. Xia, S. Wang, and R. L. Sheu}, {\em S-lemma with equality and its applications}, {Math. Program.}, 156 (2016), pp. 513--547.
	

	\bibitem{yakubovich1971} {V. A. Yakubovich}, {\em S-procedure in nonlinear control theory}, {Vestnik Leningradskogo Universiteta}, 1 (1971), pp. 62--77 {(in Russian)}.
	
	
	\bibitem{yuan-ai2017} {J. Yuan, M. Wang, W. Ai, and T. Shuai}, {\em New results on narrowing the duality gap of the extended Celis--Dennis--Tapia problem}, {SIAM J. Optim.}, 27(2) (2017), pp. 890--909.
	
	\bibitem{yuan2015} {Y. Yuan}, {\em Recent advances in trust region algorithms}, {Math. Program.}, 151 (2015), pp. 249--281.
	
\end{thebibliography}
\end{document}